\documentclass[a4paper,11pt]{article}
\usepackage[utf8]{inputenc}

\usepackage{amsmath,amssymb,amsthm}
\usepackage{graphicx,color}
\usepackage{mathtools}
\usepackage{url,hyperref}
\graphicspath{{Fig/}}

\newtheorem{Alphathm}{Theorem}

\newtheorem{thm}{Theorem}[section]
\newtheorem{lem}[thm]{Lemma}
\newtheorem{cor}[thm]{Corollary}

\theoremstyle{definition}
\newtheorem{dfn}[thm]{Definition}
\newtheorem{exl}[thm]{Example}

\theoremstyle{remark}
\newtheorem{rem}[thm]{Remark}

\def\R{\mathbb R}
\def\E{\mathbb E}
\def\H{\mathbb H}
\def\Sph{\mathbb S}
\def\X{\mathbb X}
\def\Y{\mathbb Y}
\def\dist{\operatorname{dist}}
\def\Isom{\operatorname{Isom}}
\def\cR{\mathcal{R}}
\def\im{\operatorname{im}}
\def\rk{\operatorname{rk}}
\def\RP{\mathbb R\mathrm{P}}
\def\ori{\mathrm{or}}

\def\PhiK{\Phi^\mathrm{kin}}
\def\PhiS{\Phi^\mathrm{stat}}
\def\GL{\operatorname{GL}}
\def\GS{G^{\mathrm{stat}}}
\def\GK{G^{\mathrm{kin}}}
\def\pr{\operatorname{pr}}
\def\grad{\operatorname{grad}}

\title{Statics and kinematics of frameworks in Euclidean and non-Euclidean 
geometry}
\author{Ivan Izmestiev}

\begin{document}

\maketitle

\section{Introduction}
A bar-and-joint framework is made of rigid bars connected at 
their ends by universal joints. A framework can be constrained to a plane 
or allowed to move in space. Rigidity of frameworks 
is a question of practical importance, and its mathematical study goes back to 
the 19th century. Plate-and-hinge structures such as polyhedra can be 
represented by bar-and-joint frameworks through replacement of the hinges 
by bars and rigidifying the plates with 
the help of diagonals. Thus, rigidity questions for polyhedra belong 
to the same domain.

There are two ways to approach the rigidity of a framework: through 
statics, i.~e. ability to respond to exterior loads, and through kinematics, 
i.~e. abscence of deformations. A framework is called \emph{statically rigid} 
if every system of forces with zero sum and zero moment can be compensated by 
stresses in the bars of the framework. A framework is called \emph{rigid} if it 
cannot be deformed while keeping the lengths of all bars, and 
\emph{infinitesimally rigid} if it cannot be deformed so that the lengths of 
bars stay constant in the first order. As it turns out, static rigidity is 
equivalent to the infinitesimal rigidity.

The study of statics has a long history. 
Systems of forces appear in the textbooks of Poinsot \cite{Poi03} and M\"obius 
\cite{Moe37}, and the concept of a line-bound force was one of the motivations 
for Grassman's introduction of the exterior algebra of a vector space.

Infinitesimal isometric deformations 
seem to have appeared first in the context of smooth surfaces, see \cite{Dar96} 
and references therein. In the first half of the 20th century the interest in 
the isometric deformations was stimulated by the Weyl problem, which was 
successfully solved in the 50's by Nirenberg and Alexandrov and Pogorelov. The 
Weyl problem motivated Alexandrov's works on polyhedra, in particular his 
enhanced version of the Legendre-Cauchy-Dehn rigidity theorem for convex 
polyhedra. For a survey on rigidity of smooth surfaces see \cite{Sab92, IS94, 
IS95, IMS06}, 
for rigidity of frameworks and polyhedra see \cite{Con93}.

The goal of this article is to present the fundamental notions and results 
from the rigidity theory of frameworks in the Euclidean space and to
extend them to the hyperbolic and spherical geometry. Below we
state four main theorems whose proofs are given 
in the subsequent sections.

\begin{Alphathm}
\label{thm:A}
A framework in a Euclidean, spherical, or hyperbolic space has equal numbers of
kinematic and static degrees of freedom. In particular, infinitesimal rigidity 
is equivalent to static rigidity.
\end{Alphathm}
By the number of static, respectively kinematic, degrees of freedom we mean the 
dimension of the vector space of unresolvable loads, respectively non-trivial 
infnitesimal isometric deformations. See Sections \ref{sec:Kin} and 
\ref{sec:Sta} for definitions and for a proof of Theorem \ref{thm:A}.

\begin{Alphathm}[Darboux-Sauer correspondence]
\label{thm:B}
The number of degrees of freedom of a Euclidean framework is a projective 
invariant. In particular, a framework is infinitesimally rigid if and only if 
any of its projective images is infinitesimally rigid.
\end{Alphathm}

The projective invariance of static rigidity follows from the 
interpretation of a line-bound vector (a force) in a $d$-dimensional Euclidean 
space as a bivector in $\R^{d+1}$. Linear transformations of $\R^{d+1}$ 
preserve static dependencies; at the same time they generate projective 
transformations of $\RP^d$. See Section \ref{sec:ProjStat}.

\begin{Alphathm}[Infinitesimal Pogorelov maps]
\label{thm:C}
A hyperbolic or a spherical framework has the same number of  
kinematic degrees of freedom as its geodesic Euclidean image. In particular, it 
is infinitesimally rigid if and only if its geodesic Euclidean image is.
\end{Alphathm}
By a geodesic Euclidean image of a hyperbolic framework we mean its 
representation in a Beltrami-Cayley-Klein model. A geodesic Euclidean image of 
a spherical framework is its projection from the center of the sphere to an 
affine hyperplane. Every geodesic map of an open region in the hyperbolic or 
spherical space into the Euclidean space differs from those given above by 
post-composition with a projective map.

Theorem \ref{thm:C} is related to Theorem \ref{thm:B} and is also proved in 
Section \ref{sec:ProjStat}. In the same section we describe the infinitesimal 
Pogorelov maps that send the static or kinematic vector spaces of a framework 
to the corresponding vector spaces of its geodesic image.

While the previous three theorems hold for frameworks of any combinatorics 
and in the space of any dimension, the last one is specific for frameworks in 
dimension $2$ whose underlying graph is planar.
\begin{Alphathm}[Maxwell-Cremona correspondence]
\label{thm:D}
For a framework on the sphere or in the Euclidean or hyperbolic plane based on 
a planar graph the existence of any of the following objects implies the 
existence of the other two:
\begin{enumerate}
\item
A self-stress.
\item
A reciprocal diagram.
\item
A polyhedral lift.
\end{enumerate}
\end{Alphathm}
Definitions of reciprocal diagrams and polyhedral lifts slightly differ in 
different geometries. Also, the theorem has various versions all of which are 
presented in Section \ref{sec:MC}.

The theory of isometric deformations extends to the smooth case in a quite 
straightforward way (and, as we already mentioned, probably preceded the 
kinematics of frameworks). Accordingly, there are analogs of Theorems B and C 
for smooth submanifolds of the Euclidean, hyperbolic or spherical space. In 
fact, Theorem B was proved by Darboux for smooth surfaces and only later by 
Sauer for frameworks \cite{Sau35a}. Also Theorem C was first proved by 
Pogorelov in \cite[Chapter 5]{Pog73} for smooth surfaces. On the other hand, a 
theory of statics for smooth surfaces containing an analog of Theorem A is not 
fully developed or at least not widely known. (See however the dissertation of 
Lecornu \cite{Lec80}.)

Let us set up the notation used throughout the article.
In the following, $\X^d$ stands for either $\E^d$ (Euclidean space) or 
$\Sph^d$ (spherical space) or $\H^d$ (hyperbolic space). We often view them as 
subsets of the real vector space $\R^{d+1}$:
\begin{align*}
\E^d &= \{x \in \R^{d+1} \mid x_0 = 1\},\\
\Sph^d &= \{x \in \R^{d+1} \mid \langle x, x \rangle = 1\},\\
\H^d &= \{x \in \R^{d+1} \mid \langle x, x \rangle = -1, x_0 > 0\}.
\end{align*}
Here in the second line $\langle \cdot, \cdot \rangle$ stands for the 
Euclidean, and in the third line for the Minkowski scalar product:
\[
\langle x, y \rangle = \pm x_0y_0 + x_1y_1 + \cdots + x_dy_d.
\]
Sometimes we also use $\sin_\X$ and $\cos_\X$ to denote $\sin$ and $\cos$ in 
the spherical and $\sinh$ and $\cosh$ in the hyperbolic case.


\section{Kinematics of frameworks}
\label{sec:Kin}
\subsection{Motions}
Let $\Gamma$ be a graph; we denote its vertex set by $\Gamma_0$ and its edge 
set by $\Gamma_1$. For the vertices of $\Gamma$ we use symbols $i, j$ etc.
The edges are unordered pairs of elements of $\Gamma_0$, and for brevity we 
usually write $ij$ instead of $\{i,j\} \in \Gamma_1$.

\begin{dfn}
A \emph{framework} in $\X^d$ is a graph $\Gamma$ together with a map
\[
p \colon \Gamma_0 \to \X^d, \quad i \mapsto p_i
\]
such that $p_i \ne p_j$ whenever $\{i, j\} \in \Gamma_1$. If $\X = \Sph$, then 
we additionally require $p_i \ne -p_j$ for all $\{i, j\} \in \Gamma_1$.
\end{dfn}

This is a mathematical abstraction of a bar-and-joint framework, see the 
introduction. Note that we allow intersections between the edges.

In a framework $(\Gamma, p)$, every edge receives a non-zero length 
$\dist(p_i, p_j)$. Two frameworks $(\Gamma, p)$ and $(\Gamma, p')$ with the 
same graph are called \emph{isometric}, if they have the same edge lengths: 
$\dist(p_i, p_j) = \dist(p'_i, p'_j) \text{ for all } \{i,j\} \in \Gamma_1.$
Frameworks with the same graph are called \emph{congruent}, if there is an 
ambient isometry $\Phi 
\in \Isom(\X^d)$ such that $p'_i = \Phi(p_i)$ for all $i \in \Gamma_0$.

\begin{dfn}
A framework $(\Gamma, p)$ is called \emph{globally rigid}, if every framework 
isometric to $(\Gamma, p)$ is also congruent to it.
\end{dfn}

An \emph{isometric deformation} of a framework $(\Gamma, p)$ is a continuous 
family of frameworks $(\Gamma, p(t))$ (i.~e. every $p_i(t)$ is a 
continuous path in $\X^d$), where $t \in (-\epsilon, \epsilon)$ and $p(0) = p$. 
An isometric deformation is called \emph{trivial}, if it is generated by a 
family of ambient isometries: $p_i(t) = \Phi_t(p_i)$.

\begin{dfn}
A framework $(\Gamma, p)$ is called \emph{rigid} (or locally rigid), if it has 
no non-trivial isometric deformations. A non-rigid framework is also called 
\emph{flexible}.
\end{dfn}

Clearly, global rigidity implies rigidity, but not vice versa. See Figure 
\ref{fig:RigidExl}.

\begin{figure}[htb]
\begin{center}
\includegraphics[width=.8\textwidth]{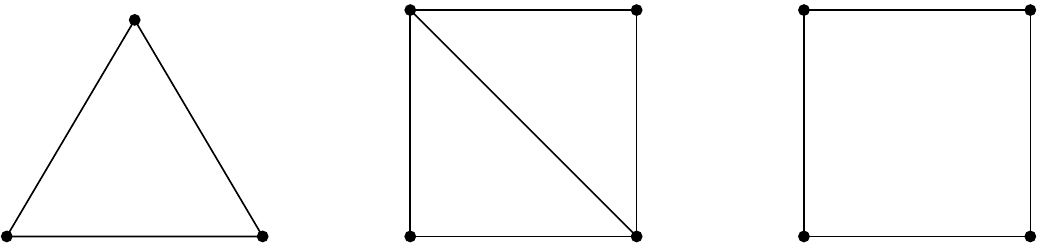}
\end{center}
\caption{Frameworks in the plane. Left: globally rigid. Middle: rigid but not 
globally rigid. Right: flexible.}
\label{fig:RigidExl}
\end{figure}


\subsection{Infinitesimal motions}
\begin{dfn}
A \emph{vector field} on a framework $(\Gamma, p)$ is a map
\[
q \colon \Gamma_0 \to T\X^d, \quad i \mapsto q_i
\]
such that $q_i \in T_{p_i}\X^d$ for all $i$. A vector field is called an 
\emph{infinitesimal isometric deformation} of $(\Gamma, p)$, if for some (and 
hence for every) smooth family of frameworks $(\Gamma, p(t))$ such that
\[
p(0) = p, \quad \left. \frac{d}{dt} \right|_{t=0} p_i(t) = q_i \text{ for all } 
i \in \Gamma_0
\]
we have
\[
\left. \frac{d}{dt} \right|_{t=0} \dist(p_i(t), p_j(t)) = 0
\]
for all $\{i,j\} \in \Gamma_1$.
\end{dfn}

Clearly, the infinitesimal isometry condition is equivalent to
\begin{equation}
\label{eqn:InfIsom}
\langle q_i, e_{ij} \rangle - \langle q_j, e_{ji} \rangle = 0,
\end{equation}
where $e_{ij} \in T_{p_i}\X^d$ is such that $\exp_{p_i}(e_{ij}) = p_j$. We will 
rewrite this in a different way.

\begin{lem}
\label{lem:IIDLinear}
A vector field $q$ is an infinitesimal isometric deformation of a 
framework $(\Gamma, p)$ if and only if
\begin{align*}
&\langle p_i - p_j, q_i - q_j \rangle = 0 &&\text{ in }\E^d;\\
&\langle p_i, q_j \rangle + \langle q_i, p_j \rangle = 0 &&\text{ in }\Sph^d 
\text{ or }\H^d.\\
\end{align*}
\end{lem}
Here $\langle p_i, q_j \rangle$ means the Euclidean, respectively Minkowski 
scalar product in $\R^{d+1}$, which makes sense if we identify $T_{p_i}\X^d$ 
with a linear subspace of $\R^{d+1}$.
\begin{proof}
This follows from \eqref{eqn:InfIsom} and
\[
 e_{ij} =
 \begin{cases}
  \frac{p_j-p_i}{\|p_j-p_i\|} &\text{ in } \E^d;\\
  \frac{p_j - \langle p_i, p_j \rangle p_i}{\sin_\X \dist(p_i,p_j)} &\text{ in 
}\Sph^d \text{ and }\H^d.
 \end{cases}
\]
\end{proof}

An infinitesimal isometric deformation is called \emph{trivial}, if there is a 
Killing field $K$ on $\X^d$ such that $q_i = K(p_i)$ for all $i$.

\begin{dfn}
A framework $(\Gamma, p)$ is called \emph{infinitesimally rigid}, if it has no 
non-trivial infinitesimal isometric deformations.
\end{dfn}

\begin{thm}
\label{thm:InfRigRig}
An infinitesimally rigid framework is rigid.
\end{thm}
For a proof, see \cite{Glu75,AR79,Con80}.

The converse of Theorem \ref{thm:InfRigRig} is false, see Figure 
\ref{fig:InfFlex}.

\begin{figure}[htb]
\begin{center}
\includegraphics[width=.5\textwidth]{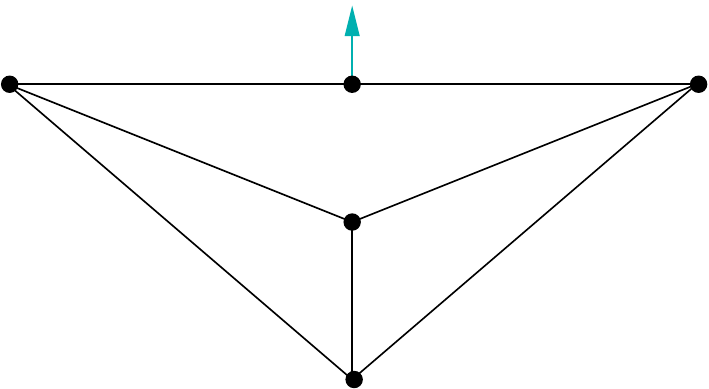}
\end{center}
\caption{A rigid but infinitesimally flexible framework.}
\label{fig:InfFlex}
\end{figure}

Similarly to the example on Figure \ref{fig:InfFlex}, one can construct a 
non-trivial infinitesimal isometric deformation for every framework contained 
in a geodesic subspace of $\X^d$ (provided that the framework has at least $3$ 
vertices). This is one of the reasons why it is convenient to consider only
\emph{spanning} frameworks: those whose vertices are not contained in a geodesic 
subspace.

Denote the set of all infinitesimal isometric deformations of a framework 
$(\Gamma, p)$ by $V(\Gamma, p)$. Due to Lemma \ref{lem:IIDLinear}, $V(\Gamma, 
p)$ is a vector space. The set of trivial infinitesimal isometric deformations 
is also a vector space; we denote it by $V_0(\Gamma, p)$. If $(\Gamma, p)$ is 
spanning, then $\dim V_0(\Gamma,p) = \frac{d(d+1)}2$.

\begin{dfn}
The dimension of the quotient space $V(\Gamma, p)/V_0(\Gamma, p)$ is called the 
number of \emph{kinematic degrees of freedom} of a framework $(\Gamma, p)$.
\end{dfn}

In particular, infinitesimally rigid frameworks are those with zero 
kinematic degrees of freedom. 

\begin{rem}
Determining whether a framework is flexible is more difficult than determining 
whether it is infinitesimally flexible: the latter is a linear problem,
the former is an algebraic one. Examples of Bricard octahedra and Kokotsakis 
polyhedra in Section \ref{sec:Exl} illustrate this.
\end{rem}

\subsection{Point-line frameworks}
\label{sec:PointLine}
A \emph{point-line framework} in $\R^2$ associates to every vertex $i$ of 
$\Gamma$ either a point $p_i$ or a line $l_i$ in $\R^2$. The edges of $\Gamma$ 
correspond to the constraints of the form
\begin{equation}
\label{eqn:PLConstr}
\dist(p_i, p_j) = \dist(p'_i, p'_j), \, \dist(p_i, l_j) 
= \dist(p'_i, l'_j), \, \angle(l_i, l_j) = \angle(l'_i, l'_j).
\end{equation}
For recent works on point-line frameworks see \cite{JO16, EJNSTW17}.

In the spherical geometry, a point-line framework is equivalent to a standard 
framework. If we replace every great circle by one of its poles, then the last 
two constraints in \eqref{eqn:PLConstr} take the form of the first one.

In the hyperbolic geometry, the pole of a line is a point in the de Sitter 
plane (the complement of the disk in the projective model of $\H^2$). 
Therefore the study of point-line frameworks in $\H^2$ can be reduced to the 
study of standard frameworks in the hyperbolic-de Sitter plane. Moreover, we 
can allow ideal points, which means assigning horocycles to some of the 
vertices of $\Gamma$ and fixing the point-horocycle, line-horocycle and 
horocycle-horocycle distances.

\subsection{Constraints counting}
\label{sec:CountConstr}
One can estimate the dimension of the space of non-congruent realizations of a 
framework by counting the constraints. If $|\Gamma_0| = n$ and $|\Gamma_1| 
= m$, then there are $m$ equations on $dn$ vertex coordinates. Besides, one 
has to subtract the dimension of the space of trivial motions, which is 
$\frac{d(d+1)}2$ for spanning frameworks. Thus, generically a framework in 
$\X^d$ with $n$ vertices and $m$ edges has $dn-m-\frac{d(d+1)}2$ degrees of 
freedom.

Of course, the above arithmetics does not make much sense without the 
combinatorics (we can put a lot of edges on a subset of the vertices, allowing 
the other vertices to fly away). Laman \cite{Lam70} has 
shown that in dimension $2$ the arithmetics and combinatorics suffice to 
characterize the generic rigidity. A graph $\Gamma$ is called a \emph{Laman 
graph} if $|\Gamma_1| = 2|\Gamma_0| - 3$ and every induced subgraph of $\Gamma$ 
with $k$ vertices has at most $2k-3$ edges.

\begin{thm}
A Laman graph is generically rigid, that is the framework $(\Gamma, p)$ is 
rigid for almost all $p$.
\end{thm}

No analog of the Laman condition is known for frameworks in higher dimensions. 
See \cite{Con03} for more details on the generic rigidity.

If all faces of a $3$-dimensional polyhedron homeomorphic to a ball are 
triangles, then its graph satisfies $|\Gamma_1| = 3|\Gamma_0| - 6$, that is the 
above count gives $0$ as the upper bound for degrees of freedom. Rigidity of 
polyhedra  is discussed in the next section.

\subsection{Frameworks and polyhedra}
One may try to generalize bar-and-joint frameworks by introducing 
panel-and-hinge structures: rigid polygons sharing pairs of sides and allowed 
to freely rotate around these sides, or even more generally $n$-dimensional 
``panels'' rotating around $(n-1)$-dimensional ``hinges''. A mathematical 
model for such an object is called a polyhedron or a polyhedral complex. 
However, there is a way to replace a polyhedral complex by a 
framework without changing its isometric deformations (global as well as local 
and infinitesimal). For this, replace every panel by a complete graph on 
its vertex set. This ``rigidifies'' the panels and leaves them the freedom to 
rotate around the hinges.

A particular class of polyhedral complexes are convex polyhedra. 
According to the Legendre-Cauchy theorem \cite{Leg94,Cau13}, a convex 
polyhedron is globally rigid among convex polyhedra. There are simple examples 
of convex polyhedra isometric to non-convex 
ones. By the Dehn theorem \cite{Dehn16} (that can also be proved by the 
Legendre-Cauchy argument), convex $3$-dimensional polyhedra are infinitesimally 
rigid.

The Legendre-Cauchy argument applies to spherical and hyperbolic convex 
polyhedra as well. This allows to prove the rigidity of convex 
polyhedra in $\X^d$ for $d > 3$ by induction: the link of a vertex of a 
$d$-dimensional convex polyhedron is a $(d-1)$-dimensional spherical 
polyhedron, and the rigidity of links implies the rigidity of the polyhedron.

A simplicial polyhedron (that is one all of whose faces are simplices) 
has the same kinematic properties as its $1$-skeleton. In a convex 
non-simplicial polyhedron we can replace every face by a complete graph as 
described in the first paragraph; but in fact a much ``lighter'' framework is 
enough to keep the polyhedron rigid. It suffices to triangulate every 
$2$-dimensional face in an arbitrary way (without adding new vertices in the 
interior of the face, vertices on the edges are all right). Again, the 
Legendre-Cauchy argument ensures the rigidity of all $3$-dimensional faces, and 
the induction applies as in the previous paragraph, \cite[Chapter 10]{Ale05}, 
\cite{Whi84a}.

As already indicated, the cone over a framework in $\Sph^d$ can be viewed as a 
panel structure (or a framework) in $\E^{d+1}$. Similarly, a 
framework in $\H^d$ leads to a framework in the $(d+1)$-dimensional Minkowski 
space.

\subsection{Averaging and deaveraging}
\label{sec:AveDeave}
There is an elegant relation between the infinitesimal and global 
flexibility. (For smooth surfaces, this idea goes back to the 19th century.)


\begin{thm}
\begin{enumerate}
\item (Deaveraging.) Let $(\Gamma, p)$ be a framework in $\X^d$ with a 
non-trivial infinitesimal 
isometric deformation $q$. Define two new frameworks $(\Gamma, p^+)$ and 
$(\Gamma, p^-)$ as follows.
\begin{align*}
&p_i^+ = p_i + q_i, \quad p_i^- = p_i - q_i
&&\text{ for } \X =\E,\\
&p_i^+ = \frac{p_i + q_i}{\|p_i +  q_i\|}, \quad p_i^- = 
\frac{p_i -  q_i}{\|p_i -  q_i\|} &&\text{ for } \X = \Sph 
\text{ or } \H.
\end{align*}
Then the frameworks $(\Gamma, p^+)$ and $(\Gamma, p^-)$ are isometric, but not 
congruent.
\item (Averaging.)
Let $(\Gamma, p')$ and $(\Gamma, p'')$ be two isometric non-congruent 
frameworks in $\X^d$. Put
\begin{align*}
&p_i = \frac{p'_i + p''_i}2, \quad q_i = \frac{p'_i - p''_i}2 &&\text{ for } 
\X=\E,\\
&p_i = \frac{p'_i + p''_i}{\|p'_i + p''_i\|}, \quad q_i = \frac{p'_i - 
p''_i}{\|p'_i + p''_i\|} &&\text{ for }\X = \Sph \text{ or } \H.
\end{align*}
Then $q$ is a non-trivial infinitesimal isometric deformation of $(\Gamma, p)$.
\end{enumerate}
\end{thm}
In the deaveraging procedure it might happen that $p_i^+ = p_j^+$ for some 
$\{i, j\} \in \Gamma_1$, so that $p^+$ is not a framework. To avoid this, one 
can replace $q$ by $cq$ for a generic $c \in \R$.

\begin{proof}
Formulas of the averaging are inverse to those of the deaveraging, and both 
statements can be proved by a direct calculation. Use that in 
the spherical and the hyperbolic cases we have $\|p_i+q_i\| = \|p_i-q_i\|$ due 
to $\langle p_i, q_i \rangle = 0$. Also $q$ is non-trivial if and only if it 
changes the distance in the first order between some $p_i$ and $p_j$ not 
connected by an edge. One can check that this is equivalent to $\dist(p_i^+, 
p_j^+) \ne \dist(p_i^-, p_j^-)$.
\end{proof}

\subsection{Examples}
\label{sec:Exl}
In Section \ref{sec:CountConstr} we spoke about generically rigid graphs. 
The most interesting examples of flexible frameworks are special realizations 
of generically rigid graphs.

\begin{exl}
\label{exl:3Prism}
[A planar framework with $9$ edges on $6$ vertices]
The frameworks on Figure \ref{fig:269} (which are combinatorially equivalent) 
are infinitesimally flexible if and 
only if the lines $a$, $b$, $c$ are concurrent, that is meeting at a 
point or parallel. This can be proved with the help of the Maxwell-Cremona 
correspondence, see Example \ref{exl:Lift3Prism}.
\end{exl}

\begin{figure}[htb]
\begin{center}
\begin{picture}(0,0)%
\includegraphics{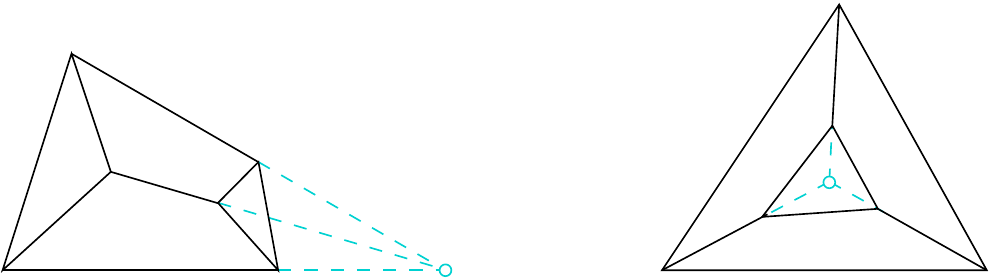}%
\end{picture}%
\setlength{\unitlength}{4144sp}%
\begingroup\makeatletter\ifx\SetFigFont\undefined%
\gdef\SetFigFont#1#2#3#4#5{%
  \reset@font\fontsize{#1}{#2pt}%
  \fontfamily{#3}\fontseries{#4}\fontshape{#5}%
  \selectfont}%
\fi\endgroup%
\begin{picture}(4524,1262)(2779,-546)
\put(3331,-466){\makebox(0,0)[lb]{\smash{{\SetFigFont{8}{9.6}{\rmdefault}{\mddefault}{\updefault}{\color[rgb]{0,0,0}$a$}%
}}}}
\put(3497,-102){\makebox(0,0)[lb]{\smash{{\SetFigFont{8}{9.6}{\rmdefault}{\mddefault}{\updefault}{\color[rgb]{0,0,0}$b$}%
}}}}
\put(3533,259){\makebox(0,0)[lb]{\smash{{\SetFigFont{8}{9.6}{\rmdefault}{\mddefault}{\updefault}{\color[rgb]{0,0,0}$c$}%
}}}}
\put(6053,-330){\makebox(0,0)[lb]{\smash{{\SetFigFont{8}{9.6}{\rmdefault}{\mddefault}{\updefault}{\color[rgb]{0,0,0}$a$}%
}}}}
\put(6950,-285){\makebox(0,0)[lb]{\smash{{\SetFigFont{8}{9.6}{\rmdefault}{\mddefault}{\updefault}{\color[rgb]{0,0,0}$b$}%
}}}}
\put(6627,339){\makebox(0,0)[lb]{\smash{{\SetFigFont{8}{9.6}{\rmdefault}{\mddefault}{\updefault}{\color[rgb]{0,0,0}$c$}%
}}}}
\end{picture}%
\end{center}
\caption{Infinitesimally flexible frameworks in the plane.}
\label{fig:269}
\end{figure}

\begin{exl}[Another planar framework with $9$ edges on $6$ vertices]
A framework based on the bipartite graph $K_{3,3}$ is infinitesimally flexible 
if and only if its vertices lie on a (possibly degenerate) conic, see Figure 
\ref{fig:Bipartite}. For a proof see \cite{BolRot80,Whi84b}. On the left hand 
side the vertices lie on a circle; the 
arrows indicate a non-trivial infinitesimal isometric deformation.
\end{exl}

\begin{figure}[htb]
\begin{center}
\includegraphics[width=.75\textwidth]{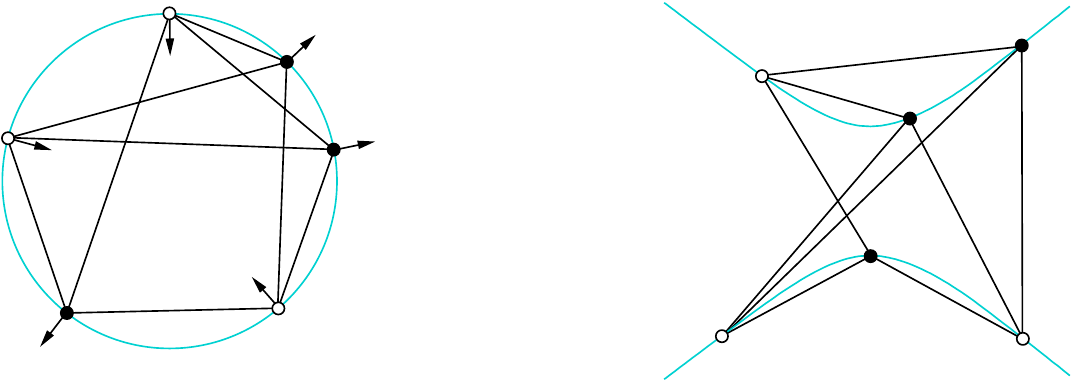}
\end{center}
\caption{Infinitesimally flexible frameworks in the plane.}
\label{fig:Bipartite}
\end{figure}

The conditions in the above two examples are projectively invariant. Besides, 
the same criteria hold for frameworks on the sphere or in the projective plane. 
(Three lines in $\H^2$ are called concurrent if they meet at a hyperbolic, 
ideal, or de Sitter point.) A non-Euclidean 
conic is one that is depicted by an affine conic in a geodesic model of 
$\Sph^2$ or $\H^2$, see \cite{Izm17}.

\begin{exl}[Bricard's octahedra and Gaifullin's cross-polytopes]
Flexible octahedra (with intersecting faces) were 
discovered and classified by Bricard \cite{Bri97}, see also \cite{Sta87, 
Naw10}. A higher-dimensional analog of the octahedron is called 
cross-polytope. Recently, flexible cross-polytopes in $\X^d$ were classified by 
Gaifullin \cite{Gai13}.
\end{exl}

\begin{exl}[Infinitesimally flexible octahedra]
While the description and classification of flexible octahedra requires quite 
some work, infinitesimally flexible octahedra can be described in a simple and 
elegant way.
\begin{quote}
Color the faces of an octahedron white and black so that adjacent faces receive 
different colors. An octahedron is infinitesimally flexible if and only if the 
planes of its four white faces meet at a point (which may lie at 
infinity). As a consequence, the planes of the white faces meet if and only if 
those of the black faces do.
\end{quote}
This theorem was proved independently by Blaschke and Liebmann 
\cite{Bla20,Lie20}. The configuration is related to the so called M\"obius 
tetrahedra: a pair of mutually inscribed tetrahedra, \cite{Moe28}.

Figure \ref{fig:InfFlexOcta} shows two examples of infinitesimally flexible 
octahedra. The one on the left is a special case of the Schoenhardt octahedron 
\cite{Schoe27}; its bases are regular triangles, and the orthogonal projection 
of one base to the other makes the triangles concentric with pairwise 
perpendicular edges.
\end{exl}

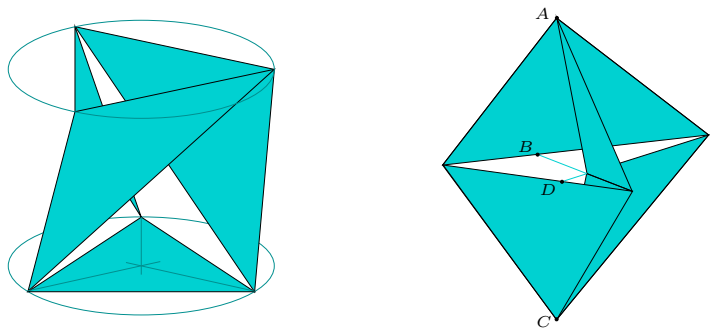
\begin{figure}[htb]
\begin{center}
\begin{picture}(0,0)%
\includegraphics{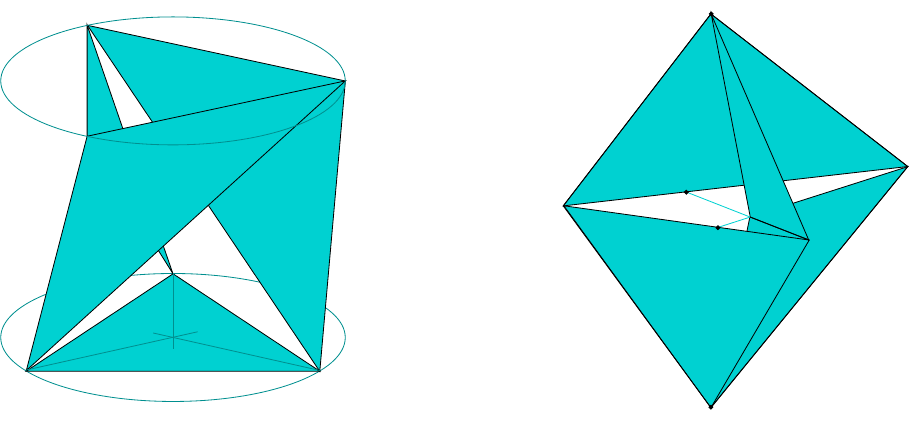}%
\end{picture}%
\setlength{\unitlength}{2072sp}%
\begingroup\makeatletter\ifx\SetFigFont\undefined%
\gdef\SetFigFont#1#2#3#4#5{%
  \reset@font\fontsize{#1}{#2pt}%
  \fontfamily{#3}\fontseries{#4}\fontshape{#5}%
  \selectfont}%
\fi\endgroup%
\begin{picture}(8313,3861)(241,-3480)
\put(6482,258){\makebox(0,0)[lb]{\smash{{\SetFigFont{6}{7.2}{\rmdefault}{\mddefault}{\updefault}{\color[rgb]{0,0,0}$A$}%
}}}}
\put(6542,-1845){\makebox(0,0)[lb]{\smash{{\SetFigFont{6}{7.2}{\rmdefault}{\mddefault}{\updefault}{\color[rgb]{0,0,0}$D$}%
}}}}
\put(6280,-1330){\makebox(0,0)[lb]{\smash{{\SetFigFont{6}{7.2}{\rmdefault}{\mddefault}{\updefault}{\color[rgb]{0,0,0}$B$}%
}}}}
\put(6502,-3425){\makebox(0,0)[lb]{\smash{{\SetFigFont{6}{7.2}{\rmdefault}{\mddefault}{\updefault}{\color[rgb]{0,0,0}$C$}%
}}}}
\end{picture}%
\end{center}
\caption{Infinitesimally flexible octahedra. On the right, the points $A$, $B$, 
$C$, $D$ must be coplanar.}
\label{fig:InfFlexOcta}
\end{figure}

Theorem \ref{thm:B} implies that infinitesimally flexible octahedra in $\Sph^3$ 
and $\H^3$ are characterized by the same criterion as those in $\E^3$. In the 
hyperbolic space, the intersection point of four planes may be ideal or 
hyperideal. In fact, even the vertices of an octahedron may be ideal or 
hyperideal. Infinitesimally flexible hyperbolic octahedra were used in 
\cite{Izm11} to construct simple examples of infinitesimally flexible 
hyperbolic cone-manifolds.

\begin{exl}[Jessen's icosahedron and its relatives]
In the $xy$-plane of $\R^3$, take the rectangle with vertices $(\pm 
1, \pm, t, 0)$, where $0<t<1$. Take two other rectangles, obtained from this 
one by $120^\circ$ and $240^\circ$ rotations around the $x=y=z$ line (which 
results in cyclic permutations of the coordinates). The convex hull of the 
twelve vertices of these rectangles is an icosahedron (a regular one for $t = 
\frac{\sqrt{5} - 1}2$). Among the edges of this icosahedron are the short sides 
of the rectangles.

Modify the $1$-skeleton by removing the short sides of 
rectangles (like the one joining $(1,t,0)$ with $(1,-t,0)$) and inserting the 
long sides (like the one joining $(1,t,0)$ with $(-1,t,0)$). The resulting 
framework $p(t)$ is the $1$-skeleton of a non-convex icosahedron. Jessen 
\cite{Jes67} gave the $t = \frac12$ non-convex icosahedron as an example 
of a closed polyhedron with orthogonal pairs of adjacent faces, but different 
from the cube. See Figure~\ref{fig:Jessen}.

\begin{figure}[htb]
\begin{center}
\includegraphics[width=.4\textwidth]{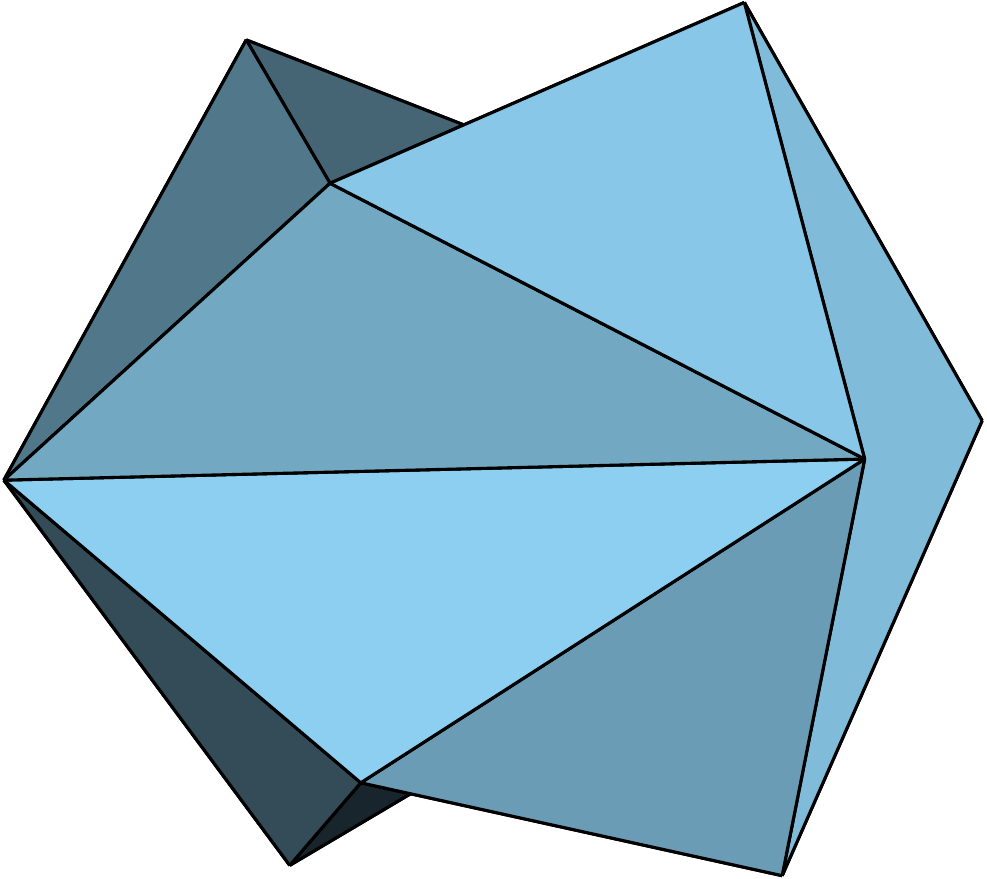}
\end{center}
\caption{Jessen's orthogonal and infinitesimally flexible icosahedron.}
\label{fig:Jessen}
\end{figure}

The framework $p(t)$ has two sorts of edges: the long sides of the rectangles, 
which have length $2$, and the sides of eight equilateral triangles, which have 
length $\sqrt{2(t^2-t+1)}$. It follows that the frameworks $p(t)$ and $p(1-t)$ 
are isometric. Note that $p(0)$ collapses to an octahedron: the map $p(t) 
\colon \Gamma_0 \to \R^3$ sends the vertices of the icosahedral graph to the 
vertices of a regular octahedron by identifying them in pairs; there are three 
pairs of edges that are mapped to three diagonals of the octahedron. At the 
same time, $p(1)$ is the graph of the cuboctahedron with square faces 
subdivided in a certain way.

Since the average of $p(t)$ and $p(1-t)$ (in the sense of Section 
\ref{sec:AveDeave}) is $p(\frac{1}{2})$, it follows that Jessen's icosahedron 
is infinitesimally flexible.

Theorem \ref{thm:B} implies that there are spherical and hyperbolic analogs 
of this construction.
\end{exl}

\begin{exl}[Kokotsakis polyhedra]
A Kokotsakis polyhedron with an $n$-gonal base is a panel structure made of a 
rigid $n$-gon, and $n$ quadrilaterals attached to its edges, and $n$ triangles 
attached between the quadrilaterals, see Figure \ref{fig:Kokotsakis}, left. 
Generically, a Kokotsakis polyhedron is 
rigid; it is flexible for certain symmetric configurations, see 
Figure \ref{fig:Kokotsakis}, right.

\begin{figure}[htb]
\begin{center}
\raisebox{.5cm}{\includegraphics[height=2.5cm]{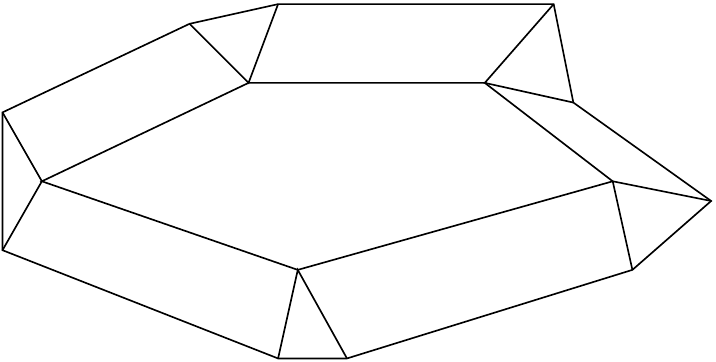}} \hspace{1cm} 
\includegraphics[height=3.5cm]{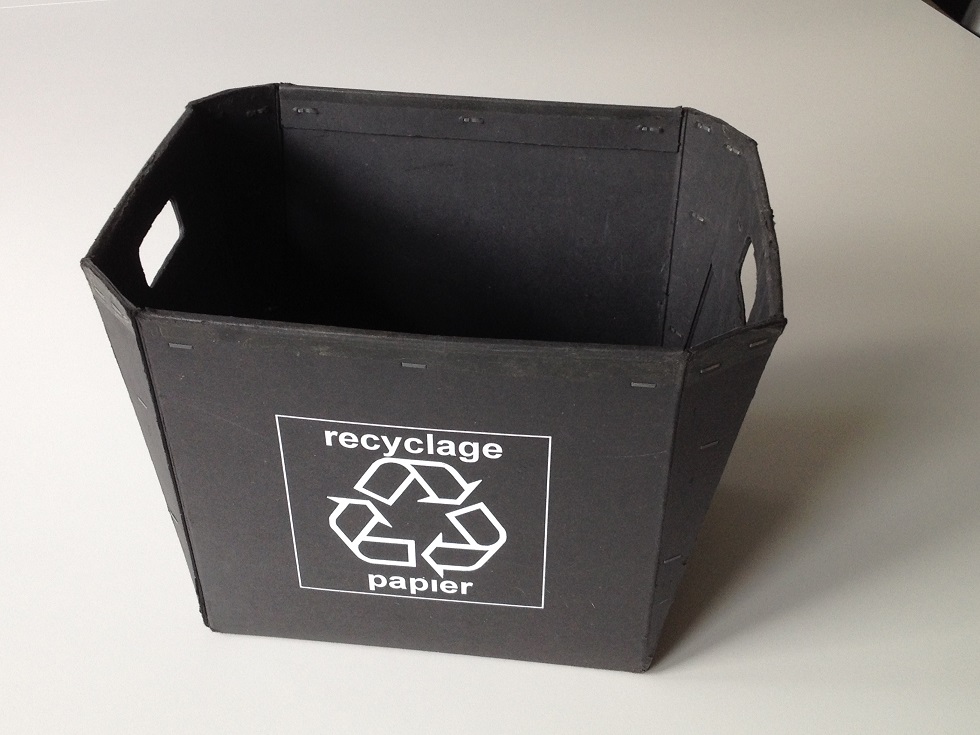}
\end{center}
\caption{Kokotsakis polyhedra.}
\label{fig:Kokotsakis}
\end{figure}

Especially interesting are the polyhedra with a quadrangular base, because of 
their relation to quad-surfaces (polyhedral surfaces made of quadrilaterals 
with four quadrilaterals around each vertex). A quad-surface is 
(infinitesimally) flexible if and only if all Kokotsakis polyhedra around its 
faces are. A famous example of a flexible quad-surface is the Miura-ori 
\cite{Ni12}.

A characterization of infinitesimally flexible Kokotsakis polyhedra was given 
by Kokotsakis in \cite{Kok33}, several flexible examples were constructed in 
\cite{SG31, Kok33}. A complete classification of flexible polyhedra with a 
quadrangular base is given in \cite{IzmKoko}.
\end{exl}

\section{Statics of frameworks}
\label{sec:Sta}
\subsection{Euclidean statics}
In the statics of a rigid body, a force is represented as a line-bound vector: 
moving the force vector along the line it spans does not change its action on 
a rigid body.
\begin{dfn}
\label{dfn:Force}
A \emph{force} in a Euclidean space is a pair $(p,f)$ with $p \in \E^d$, $f \in 
\R^d$. A \emph{system of forces} is a formal sum $\sum_i (p_i,f_i)$ that may be 
transformed according to the following rules:
\begin{enumerate}
\setcounter{enumi}{-1}
\item a force with a zero vector is a zero force:
$$
(p,0) \sim 0;
$$
\item forces at the same point can be added and scaled as usual:
$$
\lambda_1(p,f_1) + \lambda_2(p,f_2) \sim (p, \lambda_1f_1+\lambda_2f_2);
$$
\item a force may be moved along its line of action:
$$
(p,f) \sim (p + \lambda f, f).
$$
\end{enumerate}
\end{dfn}

One may check from this definition that systems of forces form a vector space 
of dimension $\frac{d(d+1)}2$.

In $\E^2$, any system of forces is equivalent either to a single force or to a 
so called ``force couple'' $(p_1, f) + (p_2, -f)$, where the vector $f$ is not 
parallel to the line through $p_1$ and $p_2$.

\begin{dfn}
\label{dfn:Load}
A \emph{load} on a Euclidean framework $(\Gamma, p)$ is a map
\[
\begin{array}{rrcl}
f \colon & \Gamma_0 & \to & \R^d,\\
& i & \mapsto & f_i.
\end{array}
\]
A load is called an \emph{equilibrium load} if the system of forces $\sum_{i \in 
\Gamma_0} (p_i, f_i)$ is equivalent to a zero force.
\end{dfn}

A rigid body responds to an equilibrium load by interior stresses that cancel 
the forces of the load. This motivates the following definition.

\begin{dfn}
\label{dfn:Stress}
A \emph{stress} on a framework $(\Gamma, p)$ is a map
\[
\begin{array}{rrcl}
w \colon & \Gamma_1 & \to & \R,\\
& ij & \mapsto & w_{ij} = w_{ji}.
\end{array}
\]
The stress $w$ is said to \emph{resolve} the load $f$ if
\begin{equation}
\label{eqn:ResLoad}
f_i = \sum_{j \in \Gamma_0} w_{ij} (p_j - p_i) \mbox{ for 
all }i \in \Gamma_0,
\end{equation}
where we put $w_{ij} = 0$ for all $ij \notin \Gamma_1$.
\end{dfn}

We denote the vector space of equilibrium loads by $F(\Gamma,p)$, and the vector space of resolvable loads by $F_0(\Gamma,p)$. It is easy to see that every resolvable load is an equilibrium load: $F_0(\Gamma, p) \subset F(\Gamma, p)$.

\begin{dfn}
The dimension of the quotient space $F(\Gamma, p)/F_0(\Gamma, p)$ is called the number of \emph{static degrees of freedom} of the framework $(\Gamma, p)$.

The framework $(\Gamma, p)$ is called \emph{statically rigid} if it has zero 
static degrees of freedom, i.~e. if every equilibrium load can be 
resolved.
\end{dfn}

\subsection{Non-euclidean statics}
\begin{dfn}
Let $\X^d = \Sph^d$ or $\H^d$. A \emph{force} in $\X^d$ is an element of the 
tangent bundle $T\X^d$. We write it as a pair $(p, f)$ with $p \in \X^d$ 
and $f \in T_p\X^d$.

A \emph{system of forces} is a formal sum of forces that 
may be transformed according to the rules of Definition \ref{dfn:Force}, where 
the formula in the rule 2) is replaced by $(p, f) \sim (\exp_p(\lambda f), 
\tau(f))$ with $\tau(f)$ being the result of the parallel transport of $f$ 
along the geodesic from $p$ to $\exp_p(\lambda f)$.
\end{dfn}

A system of forces on $\Sph^2$ is always equivalent to a single force; a system 
of forces on $\H^2$ is equivalent to either a single force, or an ideal force
couple or a hyperideal force couple.

\begin{dfn}
A load on a framework $(\Gamma, p)$ in $\X^d$ is a map
\[
f \colon \Gamma_0 \to  T\X^d, \quad f_i \in T_{p_i}\X^d.
\]
A load is called an \emph{equilibrium load} if the system of forces $\sum_{i 
\in \Gamma_0} (p_i, f_i)$ is equivalent to a zero force.
\end{dfn}

In the above definitions, $\X^d$ can also stand for $\E^d$. The canonical 
isomorphisms $T_x\E^d \cong \R^d$ result in simplified formulations given in 
the preceding section.

As in the Euclidean case, a stress on a framework in $\X^d$ is a map $w \colon 
\Gamma_1 \to \R$. A stress $w$ \emph{resolves a load} $f$ if
\[
f_i = \sum_{j \in \Gamma_0} w_{ij} \dist(p_i,p_j) e_{ij},
\]
where $e_{ij} \in T_{p_i}\X^d$ is such that $\exp_{p_i}(e_{ij}) = p_j$.
The following lemma gives an alternative description of the stress resolution.

\begin{lem}
\label{lem:StressExtr}
A stress $w$ resolves a load $f$ on a framework $(\Gamma, p)$ in $\X^d = 
\Sph^d$ or $\H^d$ if and only if for every $i \in \Gamma_0$ we have
\[
f_i - \sum_{j \in \Gamma_0} \lambda_{ij} p_j \parallel p_i,
\]
where $\lambda_{ij} = w_{ij} \frac{\dist(p_i,p_j)}{\sin_{\X} \dist(p_i,p_j)}$. 
Here $f_i, p_i \in \R^{d+1}$ via $\X^d \subset \R^{d+1}$.
\end{lem}
\begin{proof}
Follows from the identity
\[
p_j - \cos_{\X} \dist(p_i, p_j) p_i = p_j - \langle p_i, p_j \rangle p_i = 
\sin_{\X}\dist(p_i,p_j) \cdot e_{ij}.
\]
\end{proof}

\subsection{Equivalence of static and infinitesimal rigidity}
\label{sec:EqStInf}
Define a pairing between vector fields and loads on a framework $(\Gamma,p)$:
\begin{equation}
\label{eqn:Pairing}
\langle q, f \rangle = \sum_{i \in \Gamma_0} \langle q_i, f_i \rangle.
\end{equation}
This pairing is non-degenerate and therefore induces a duality between the 
space of vector fields and the space of loads.

\begin{lem}[Principles of virtual work]
\label{lem:Dual}
Under the pairing \eqref{eqn:Pairing},
\begin{enumerate}
\item
the space of infinitesimal motions is the annihilator of the space of 
resolvable loads:
$$
V(\Gamma, p) = F_0(\Gamma,p)^\circ;
$$
\item
the space of trivial infinitesimal motions is the annihilator of the space of 
equilibrium loads:
$$
V_0(\Gamma, p) = F(\Gamma,p)^\circ.
$$
\end{enumerate}
\end{lem}
A proof in the Euclidean case can be found in \cite{Izm09}; it transfers to the 
spherical and the hypebolic cases.

As a consequence, the pairing \eqref{eqn:Pairing} induces an isomorphism
\begin{equation}
\label{eqn:StatKinDual}
V(\Gamma, p)/V_0(\Gamma, p) \cong \left( F(\Gamma, 
p)/F_0(\Gamma, p) \right)^*
\end{equation}
which implies Theorem \ref{thm:A}.
%

The statics of a Euclidean framework is formulated in purely linear terms: 
loads and stresses on a framework correspond to loads and stresses on its 
affine image. Together with Theorem \ref{thm:A} this leads to the 
following conclusion, which is a special case of Theorem \ref{thm:B}.
\begin{cor}
\label{cor:InfAff}
The number of kinematic degrees of freedom of a Euclidean framework is an affine 
invariant.
In particular, an affine image of an infinitesimally rigid framework is 
infinitesimally rigid.
\end{cor}

\begin{dfn}
\label{dfn:RigMat}
The \emph{rigidity matrix} of a Euclidean framework $(\Gamma, p)$ is a 
$\Gamma_1 \times \Gamma_0$ matrix with vector entries:
\[
\cR(\Gamma,p) = \; \scriptstyle{ij}
              \stackrel{\scriptstyle{i}}{\left(\begin{array}{ccc}
              & \vdots & \\
              \cdots & p_i - p_j & \cdots \\
              & \vdots &
             \end{array}\right)}.
\]
It has the pattern of the edge-vertex incidence matrix of the graph $\Gamma$, with $p_i - p_j$ on the intersection of the row $ij$ and the column $i$.
\end{dfn}

The rows of $\cR(\Gamma,p)$ span the space $F_0(\Gamma,p)$. The following 
proposition is a reformulation of the first principle of virtual work.

\begin{lem}
\label{lem:RigMat}
Consider $\cR(\Gamma,p)$ as the matrix of a map $(\R^d)^{\Gamma_0} \to \R^{\Gamma_1}$. Then the following holds:
$$
\begin{array}{lcl}
\ker \cR(\Gamma,p) & = & V(\Gamma, p);\\
\im \cR(\Gamma,p)^\top & = & F_0(\Gamma,p).
\end{array}
$$
\end{lem}

\begin{cor}
A framework $(\Gamma,p)$ is infinitesimally rigid if and only if
\[
\rk \cR(\Gamma,p) = d\, |\Gamma_0| - \binom{d+1}2.
\]
\end{cor}

\section{Projective statics and kinematics}
\subsection{Projective statics}
\label{sec:ProjStat}
For $\X^d = \E^d$, $\Sph^d$ or $\H^d$ associate to a force $(p,f)$ in $\X^d$ a 
bivector in $\R^{d+1}$:
\begin{equation}
\label{eqn:ForceBivector}
(p, f) \mapsto p \wedge f.
\end{equation}
We use the canonical embeddings $\X^d \subset \R^{d+1}$ 
that allow to view a point $p$ and a vector $f$ as vectors in $\R^{d+1}$.

\begin{lem}
\label{lem:SystForceExt}
The map \eqref{eqn:ForceBivector} extends to an isomorphism between the space 
of systems of forces on $\X^d$ and the second exterior power 
$\Lambda^2(\R^{d+1})$.
\end{lem}
The equivalence relations from Definition \ref{dfn:Force} ensure that a linear 
extension is well-defined. For a proof of its bijectivity, see \cite{Izm09}.

The above observation motivates the following definitions.

\begin{dfn}
A \emph{projective framework} is a graph $\Gamma$ together with a map
\[
\pi \colon \Gamma_0 \to \RP^d, \quad i \mapsto \pi_i,
\]
such that $\pi_i \ne \pi_j$ for $ij \in \Gamma_1$.
\end{dfn}

We say that $\phi \in \Lambda^2(\R^{d+1})$ is \emph{divisible} by a vector $v$, 
if $\phi = v \wedge w$ for some vector $w$. Similarly, we say that $\phi$ is 
divisible by $\pi \in \RP^d$, if $\phi$ is divisible by a representative of 
$\pi$.

\begin{dfn}
A \emph{load on a projective framework} $(\Gamma, \pi)$ is a map
\[
\phi \colon \Gamma_0 \to \Lambda^2(\R^{d+1}), \quad i \mapsto \phi_i,
\]
that sends every vertex $i$ to a bivector divisible by $\pi_i$.
An \emph{equilibrium load} is one that satisfies
\[
\sum_{i \in \Gamma_0} \phi_i = 0.
\]
\end{dfn}

\begin{dfn}
Denote by $\Gamma_1^{\ori}$ the set of oriented edges of the graph $\Gamma$.
A \emph{stress on a projective framework} $(\Gamma, \pi)$ is a map
\[
\Omega \colon \Gamma_1^{\ori} \to \Lambda^2(\R^{d+1}), \quad (i,j) \mapsto 
\omega_{ij}
\]
such that $\omega_{ij}$ is divisible by both $\pi_i$ and $\pi_j$, and 
$\omega_{ij} = - \omega_{ji}$.

A stress $\Omega$ is said to \emph{resolve a load} $\phi$ if
\[
\phi_i = \sum_{j \in \Gamma_0} \omega_{ij}.
\]
\end{dfn}

The \emph{projectivization} of a framework $(\Gamma, p)$ in $\X^d$ is obtained 
by composing $p$ with the inclusion $\X^d \subset \R^{d+1}$ and the projection 
$\R^{d+1} \setminus \{0\} \to \RP^d$. The following lemma is straightforward.

\begin{lem}
\label{lem:ProjStat}
The map \eqref{eqn:ForceBivector} sends bijectively the 
equilibrium, respectively resolvable, loads on a framework in $\X^d$ to the 
equilibrium, respectively resolvable, loads on its projectivization.
\end{lem}

Theorems \ref{thm:B} and \ref{thm:C} are immediate 
corollaries of Lemma \ref{lem:ProjStat}.

\begin{proof}[Proof of Theorem \ref{thm:B}]
Two frameworks in $\E^d$ are 
projective images of one another if and only if their projectivizations are 
related by a linear isomorphism of $\R^{d+1}$. A linear map sends equilibrium 
loads to equilibrium ones, and resolvable to resolvable ones.
\end{proof}

It seems that Theorem \ref{thm:B} was first proved by Rankine \cite{Ran63} in 
1863. He stated that the static rigidity is projective invariant but did not 
give the details, just saying that ``... theorems discovered by Mr. Sylvester 
... obviously give at once the solution of the question''. The first detailed 
accounts are \cite{Lie20} (for a special case $|\Gamma_1| = d |\Gamma_0| - 
\frac{d(d+1)}2$) and \cite{Sau35a}.

\begin{proof}[Proof of Theorem \ref{thm:C}]
A Euclidean framework and its geodesic spherical or hyperbolic image 
have the same projectivizations. Hence the maps \eqref{eqn:ForceBivector} 
yield an isomorphism between the spaces of their equilibrium/resolvable loads.
\end{proof}

\subsection{Static and kinematic Pogorelov maps}
Let a framework $(\Gamma, p)$ in $\E^d$ and a projective map $\Phi \colon \RP^d 
\to \RP^d$ be given such that the image of $\Phi \circ p$ is contained in 
$\E^d$. (Here $\RP^d$ is a projective completion of $\E^d$). Lemma 
\ref{lem:ProjStat} does not only show that the spaces of equilibrium 
modulo resolvable loads of $(\Gamma, p)$ and $(\Gamma, \Phi \circ p)$ have the 
same dimension, but also establishes a canonical up to a scalar factor 
isomorphism between these spaces. Through the static-kinematic duality 
from Section \ref{sec:EqStInf} this also yields an isomorphism between the 
spaces of infinitesimally isometric modulo trivial motions.

The situation is similar with the geodesic correspondence between frameworks 
in different geometries. The kinematic isomorphisms were 
described by Pogorelov in \cite[Chapter 5]{Pog73} together with the maps that 
associate to a pair of isometric polyhedra in one geometry a pair of isometric 
polyhedra with the same combinatorics in the other geometry (related to 
the kinematic isomorphism via the averaging procedure, see Section 
\ref{sec:AveDeave}). We will use the term \emph{Pogorelov maps} in each of the 
above situations.

\begin{dfn}
Let $X \subset \X^d$ and $Y \subset \Y^d$, where $\X, \Y \in \{\E, \Sph, \H\}$, 
and let $\Phi \colon X \to Y$ be a geodesic map. A fiberwise linear map $\PhiS 
\colon TX \to TY$ with $\PhiS(T_pX) \subset T_{\Phi(p)}Y$ is called a 
\emph{static Pogorelov map} associated with $\Phi$ if for every framework 
$(\Gamma, p)$ in $X$ the following two conditions are satisfied:
\begin{itemize}
 \item a load $f$ on $(\Gamma, p)$ is in equilibrium if and only if the load 
$\PhiS \circ f$ on the framework $(\Gamma, \Phi \circ p)$ is in equilibrium;
 \item a load $f$ on $(\Gamma, p)$ is resolvable if and only if the load 
$\PhiS \circ f$ on the framework $(\Gamma, \Phi \circ p)$ is resolvable.
\end{itemize}

A fiberwise linear map $\PhiK 
\colon TX \to TY$ with $\PhiK(T_pX) \subset T_{\Phi(p)}Y$ is called a 
\emph{kinematic Pogorelov map} associated with $\Phi$ if for every framework 
$(\Gamma, p)$ in $X$ the following two conditions are satisfied:
\begin{itemize}
 \item a vector field $q$ on $(\Gamma, p)$ is an infinitesimal isometric 
deformation if and only if the vector field 
$\PhiK \circ q$ on $(\Gamma, \Phi \circ p)$ is an infinitesimal isometric 
deformation;
 \item a vector field $q$ on $(\Gamma, p)$ is a trivial infinitesimal isometric 
deformation if and only if the vector field 
$\PhiK \circ q$ on  $(\Gamma, \Phi \circ p)$ is a trivial infinitesimal 
isometric deformation.
\end{itemize}
\end{dfn}

\begin{rem}
The last condition on a kinematic Pogorelov map means that $\PhiK$ sends 
Killing 
fields on $X$ to Killing fields on $Y$. For an intrinsic approach to the 
Pogorelov maps defined for Riemannian metrics with the same geodesics, see 
\cite[Section 4.3]{FS}.
\end{rem}

%

\begin{lem}
\label{lem:PhiSPhiK}
If $\PhiS$ is a static Pogorelov map associated with $\Phi$, then 
$((\PhiS)^{-1})^*$ is a kinematic Pogorelov map associated with $\Phi$.
\end{lem}
\begin{proof}
Follows from
\[
\langle ((\PhiS)^{-1})^*(q), \PhiS(f) \rangle = \langle q, (\PhiS)^{-1} \circ 
\PhiS (q) \rangle = \langle q, f \rangle
\]
and from Lemma \ref{lem:Dual}.
\end{proof}

\subsection{Pogorelov maps for affine and projective transformations}
\begin{thm}
Let $\Phi \colon \E^d \to \E^d$ be an affine transformation with the linear 
part $A = d\Phi \in \GL(n, \R)$. Then
\[
\PhiS = A, \quad \PhiK = (A^{-1})^*
\]
are static and kinematic Pogorelov maps for $\Phi$.
\end{thm}
\begin{proof}
Equivalence relation in Definition \ref{dfn:Force} is affinely invariant. 
Therefore $f$ is an equilibrium load on $(\Gamma, p)$ if and only if $A \circ 
f$ is an equilibrium load on $(\Gamma, \Phi \circ p)$. When $(\Gamma, p)$ is 
transformed by $\Phi$, the right hand side of \eqref{eqn:ResLoad} is 
transformed by $A$. Therefore a stress that resolves $f$ also resolves $A \circ 
f$.
\end{proof}

%

\begin{thm}
Let
\[
\Phi \colon \E^d \setminus L \to \E^d \setminus L'
\]
be a projective transformation, where $L$ is the hyperplane sent to infinity, 
and $L'$ is the image of the hyperplane at infinity.
Denote by $h_L(p)$ the distance from a point $p \in \E^d$ to the hyperplane 
$L$. Then
\[
\PhiS_p = h_L^2(p) \cdot d\Phi_p,\quad \PhiK_p = h_L^{-2}(p) \cdot 
((d\Phi_p)^*)^{-1}
\]
are static and kinematic Pogorelov maps for $\Phi$.
%
\end{thm}
\begin{proof}
A projective transformation $\Phi$ consists of a linear 
transformation $M \in \GL(d+1, \R)$ restricted to $\E^d$ followed by the 
central projection from the origin to $\E^d$.
We need to compose the map 
\eqref{eqn:ForceBivector} with $M_* \colon \Lambda^2(\R^{d+1}) \to 
\Lambda^2(\R^{d+1})$ and then with the inverse of 
\eqref{eqn:ForceBivector}.

The map \eqref{eqn:ForceBivector} followed by $M_*$ transforms 
a force $(p,v)$ as follows:
\[
(p,v) \mapsto p \wedge v = p \wedge (p+v) \mapsto M(p) \wedge M(p+v).
\]
We have
\[
M(p) = \frac{h_{M(L)}(M(p))}{\dist(\E^d \cap M(\E^d), M(L)} \Phi(p) = c \cdot 
h_L(p) \cdot \Phi(p)
\]
for some $c \in \R$, where the distances are taken with a sign, see Figure 
\ref{fig:ProjPogMap}. It 
follows that
\[
M(p) \wedge M(p+v) = c^2 \cdot h_L(p) \cdot h_L(p+v) \cdot \Phi(p) \wedge 
\Phi(p+v).
\]

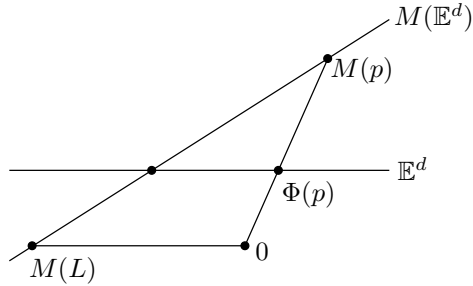
\begin{figure}[htb]
\begin{center}
\begin{picture}(0,0)%
\includegraphics{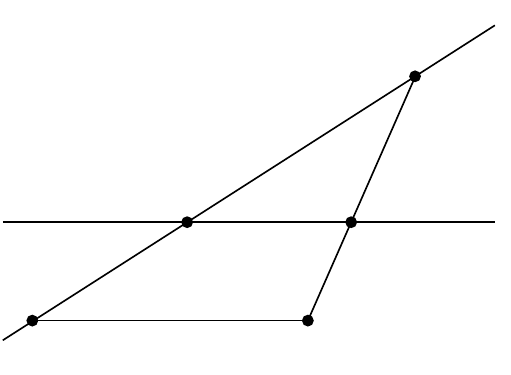}%
\end{picture}%
\setlength{\unitlength}{4144sp}%
\begingroup\makeatletter\ifx\SetFigFont\undefined%
\gdef\SetFigFont#1#2#3#4#5{%
  \reset@font\fontsize{#1}{#2pt}%
  \fontfamily{#3}\fontseries{#4}\fontshape{#5}%
  \selectfont}%
\fi\endgroup%
\begin{picture}(2327,1711)(-11,-759)
\put(1460,-594){\makebox(0,0)[lb]{\smash{{\SetFigFont{10}{12.0}{\rmdefault}{\mddefault}{\updefault}{\color[rgb]{0,0,0}$0$}%
}}}}
\put(1904,498){\makebox(0,0)[lb]{\smash{{\SetFigFont{10}{12.0}{\rmdefault}{\mddefault}{\updefault}{\color[rgb]{0,0,0}$M(p)$}%
}}}}
\put(119,-695){\makebox(0,0)[lb]{\smash{{\SetFigFont{10}{12.0}{\rmdefault}{\mddefault}{\updefault}{\color[rgb]{0,0,0}$M(L)$}%
}}}}
\put(2280,805){\makebox(0,0)[lb]{\smash{{\SetFigFont{10}{12.0}{\rmdefault}{\mddefault}{\updefault}{\color[rgb]{0,0,0}$M(\E^d)$}%
}}}}
\put(1616,-243){\makebox(0,0)[lb]{\smash{{\SetFigFont{10}{12.0}{\rmdefault}{\mddefault}{\updefault}{\color[rgb]{0,0,0}$\Phi(p)$}%
}}}}
\put(2301,-118){\makebox(0,0)[lb]{\smash{{\SetFigFont{10}{12.0}{\rmdefault}{\mddefault}{\updefault}{\color[rgb]{0,0,0}$\E^d$}%
}}}}
\end{picture}%

\end{center}
\caption{Computing the Pogorelov map for a projective transformation.}
\label{fig:ProjPogMap}
\end{figure}

Applying the inverse of \eqref{eqn:ForceBivector} we see that the vector $v$ at 
$p$ is transformed to the vector
\[
c^2 \cdot h_L(p) \cdot h_L(p+v) \cdot (\Phi(p+v) - 
\Phi(p))
\]
at $\Phi(p)$. By construction, this transformation is linear in $v$. Therefore 
it does not change if we replace $v$ by $tv$ and take the derivative with 
respect to $t$ at $t=0$. This derivative equals $c^2 h_L^2(p) d\Phi_p(v)$. This 
proves the formula for $\PhiS$. The formula for $\PhiK$ follows from Lemma 
\ref{lem:PhiSPhiK}.
\end{proof}

\subsection{Pogorelov maps for geodesic projections of $\Sph^d$ and $\H^d$}

\begin{thm}
Let $G \colon \E^d \to X$ be the projection from the origin of 
$\R^{d+1}$, where $X = \Sph_+^d$ or $X = \H^d$.

Then the Pogorelov maps for a Euclidean framework $(\Gamma, p)$ and its 
spherical, respectively hyperbolic, image $(\Gamma, G \circ p)$ are given by
\[
\GS_p = \|p\| \cdot dG_p, \quad \GK_p = \frac1{\|p\|} (dG_p^*)^{-1}.
\]
Here $\| \cdot \|$ denotes the Euclidean, respectively Minkowski, norm in 
$\R^{d+1}$.
\end{thm}
Note that in the spherical case at the point $e_0$ (the tangency point of 
$X$ with $\E^d$) we have $\GS_{e_0} = dG_{e_0}$. In the hyperbolic case we have 
$\GS_{e_0} = -dG_{e_0}$, so one might want to change the sign in the formulas.

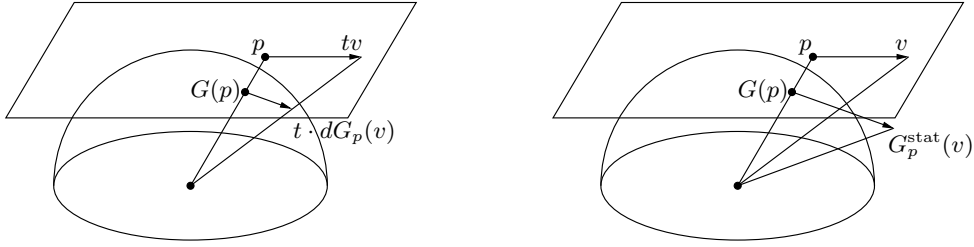
\begin{figure}[htb]
\begin{center}
\begin{picture}(0,0)%
\includegraphics{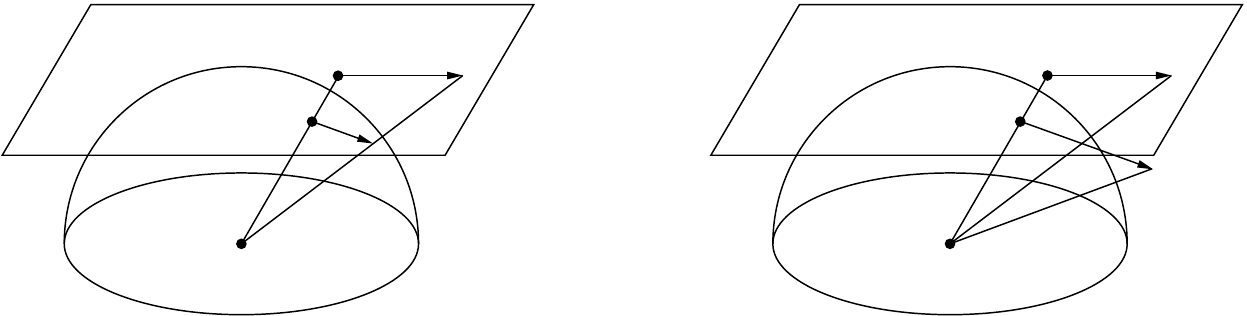}%
\end{picture}%
\setlength{\unitlength}{3729sp}%
\begingroup\makeatletter\ifx\SetFigFont\undefined%
\gdef\SetFigFont#1#2#3#4#5{%
  \reset@font\fontsize{#1}{#2pt}%
  \fontfamily{#3}\fontseries{#4}\fontshape{#5}%
  \selectfont}%
\fi\endgroup%
\begin{picture}(6324,1594)(-326,-428)
\put(1306,839){\makebox(0,0)[lb]{\smash{{\SetFigFont{9}{10.8}{\rmdefault}{\mddefault}{\updefault}{\color[rgb]{0,0,0}$p$}%
}}}}
\put(1891,839){\makebox(0,0)[lb]{\smash{{\SetFigFont{9}{10.8}{\rmdefault}{\mddefault}{\updefault}{\color[rgb]{0,0,0}$tv$}%
}}}}
\put(1576,254){\makebox(0,0)[lb]{\smash{{\SetFigFont{9}{10.8}{\rmdefault}{\mddefault}{\updefault}{\color[rgb]{0,0,0}$t \cdot dG_p(v)$}%
}}}}
\put(5536,839){\makebox(0,0)[lb]{\smash{{\SetFigFont{9}{10.8}{\rmdefault}{\mddefault}{\updefault}{\color[rgb]{0,0,0}$v$}%
}}}}
\put(5491,164){\makebox(0,0)[lb]{\smash{{\SetFigFont{9}{10.8}{\rmdefault}{\mddefault}{\updefault}{\color[rgb]{0,0,0}$\GS_p(v)$}%
}}}}
\put(4906,839){\makebox(0,0)[lb]{\smash{{\SetFigFont{9}{10.8}{\rmdefault}{\mddefault}{\updefault}{\color[rgb]{0,0,0}$p$}%
}}}}
\put(901,524){\makebox(0,0)[lb]{\smash{{\SetFigFont{9}{10.8}{\rmdefault}{\mddefault}{\updefault}{\color[rgb]{0,0,0}$G(p)$}%
}}}}
\put(4501,524){\makebox(0,0)[lb]{\smash{{\SetFigFont{9}{10.8}{\rmdefault}{\mddefault}{\updefault}{\color[rgb]{0,0,0}$G(p)$}%
}}}}
\end{picture}%

\end{center}
\caption{Computing the Pogorelov map for a geodesic projection to the sphere.}
\label{fig:DiffStat}
\end{figure}

\begin{proof}
To compute the image of 
$v \in T_p\E^d$ under the differential $dG_p$, project the geodesic 
$p+tv$ in $\E^d$ to $X$. Then $dG_p(v)$ is the velocity vector of 
the projected curve at $t=0$, see Figure \ref{fig:DiffStat}, left, than 
illustrates the case of the sphere. On the other hand, the image of 
$v$ under the static Pogorelov map is determined by
\[
G(p) \wedge \GS_p(v) = p \wedge v.
\]
Hence both $dG_p(v)$ and $\GS_p(v)$ are linear combinations of $p$ and $v$ 
tangent to $\Sph^d$. It follows that these two vectors are collinear:
\[
\GS_p(v) = \lambda(p,v) \cdot dG_p(v), \quad \lambda(p,v) \in \R.
\]
If the images of every vector under two linear maps are collinear, then these 
maps are scalar multiples of each other. Thus $\lambda$ depends 
on $p$ only:
\[
\GS_p = \lambda(p) \cdot dG_p.
\]
For small $t$, the ratio of the areas of the triangles on Figure 
\ref{fig:DiffStat}, left, is equal to $\|p\|$. Hence
\[
G(p) \wedge dG_p(v) = \frac{1}{\|p\|} p \wedge v,
\]
which implies the first formula of the theorem. The second formula follows from 
the duality between infinitesimal deformations and loads.
\end{proof}

%

\section{Maxwell-Cremona correspodence}
\label{sec:MC}
\subsection{Planar $3$-connected graphs, polyhedra, and duality}
\label{sec:MCSetup}
A graph is called \emph{$3$-connected} if it is connected, has at least $4$ 
vertices, and remains connected after removal of any two of its vertices. In 
particular, every vertex of a $3$-connected graph has degree at least $3$.

Planar $3$-connected graphs have very nice properties. First, by a result 
of Whitney \cite{Whi32}, their embeddings into $\Sph^2$ split in two isotopy 
classes that differ by an orientation-reversing diffeomorphism of $\Sph^2$. 
Second, by the 
Steinitz theorem \cite{Ste22, Zie95}, a graph is planar and $3$-connected if 
and only if it is isomorphic to the skeleton of some convex $3$-dimensional 
polyhedron. Whitney's theorem implies that for a planar $3$-connected graph 
$\Gamma$ there is a well-defined set of \emph{faces} $\Gamma_2$. Geometrically, 
a face is a connected component of $\Sph^2 \setminus \phi(\Gamma)$, where $\phi$ 
is an embedding of $\Gamma$; combinatorially it is the set of vertices on the 
boundary of such a component. We call 
$(\alpha, i)$ with $\alpha \in \Gamma_2$, 
$i \in \Gamma_0$ and $i \in \alpha$ an \emph{incident pair}. Choice of an 
isotopy 
class of an embedding $\Gamma \to \Sph^2$ and of an orientation of $\Sph^2$ 
induces a cyclic order on the set of vertices incident to a face.

The \emph{dual graph} $\Gamma^*$ of a planar $3$-connected graph $\Gamma$ can 
be constructed from an embedding $\Gamma \to \Sph^2$ by choosing a point inside 
every face and joining every pair of points whose corresponding faces share an 
edge. The graph $\Gamma^*$ is also planar and $3$-connected, and its dual is 
again $\Gamma$.
If an edge $ij$ of $\Gamma$ separates the faces $\alpha$ and $\beta$, then we 
say that $(\alpha\beta, ij)$ is a \emph{dual pair of edges}. Choose an 
isotopy class of embeddings $\Gamma \to \Sph^2$ and fix an orientation of 
$\Sph^2$. Then we say that the pair $(\alpha\beta, ij)$ is \emph{consistently 
oriented} if the face $\alpha$ lies on the right from the edge $ij$ directed 
from $i$ to $j$, see Figure \ref{fig:DualPair}. Changing the order of $i$ and 
$j$ or of $\alpha$ and $\beta$ 
transforms an inconsistently oriented pair into a consistently oriented one.

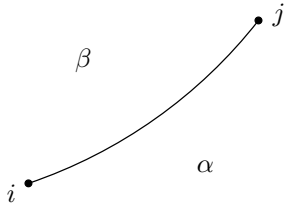
\begin{figure}[htb]
\begin{center}
\begin{picture}(0,0)%
\includegraphics{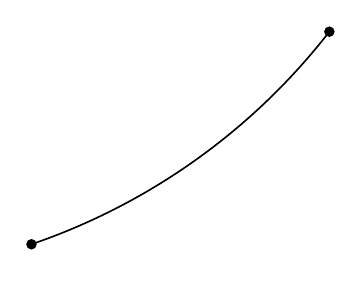}%
\end{picture}%
\setlength{\unitlength}{4144sp}%
\begingroup\makeatletter\ifx\SetFigFont\undefined%
\gdef\SetFigFont#1#2#3#4#5{%
  \reset@font\fontsize{#1}{#2pt}%
  \fontfamily{#3}\fontseries{#4}\fontshape{#5}%
  \selectfont}%
\fi\endgroup%
\begin{picture}(1605,1282)(166,-296)
\put(181,-241){\makebox(0,0)[lb]{\smash{{\SetFigFont{10}{12.0}{\rmdefault}{\mddefault}{\updefault}{\color[rgb]{0,0,0}$i$}%
}}}}
\put(1756,839){\makebox(0,0)[lb]{\smash{{\SetFigFont{10}{12.0}{\rmdefault}{\mddefault}{\updefault}{\color[rgb]{0,0,0}$j$}%
}}}}
\put(1306,-61){\makebox(0,0)[lb]{\smash{{\SetFigFont{10}{12.0}{\rmdefault}{\mddefault}{\updefault}{\color[rgb]{0,0,0}$\alpha$}%
}}}}
\put(586,569){\makebox(0,0)[lb]{\smash{{\SetFigFont{10}{12.0}{\rmdefault}{\mddefault}{\updefault}{\color[rgb]{0,0,0}$\beta$}%
}}}}
\end{picture}%
\end{center}
\caption{A consistently oriented dual pair.}
\label{fig:DualPair}
\end{figure}

\subsection{Maxwell-Cremona theorem}
For convenience we identify in this section $\E^2$ with $\R^2$ by choosing an 
origin.
\begin{dfn}
\label{dfn:RecDiagEuc}
Let $(\Gamma, p)$ be a framework in $\R^2$ with a planar $3$-connected graph 
$\Gamma$. A \emph{reciprocal diagram} for $(\Gamma, p)$ is a framework 
$(\Gamma^*, m)$ such that dual edges are perpendicular to each other:
\[
m_\beta - m_\alpha \perp p_j - p_i
\]
whenever the edge $ij$ of $\Gamma$ separates the faces $\alpha$ and $\beta$.
\end{dfn}


\begin{dfn}
\label{dfn:PolLift}
Let $(\Gamma, p)$ be a framework in $\R^2$ with a planar $3$-connected graph 
$\Gamma$ and such that for every face $\alpha \in \Gamma_2$ the points 
$\{p_i \mid i \in \alpha\}$ are not collinear. A \emph{vertical polyhedral 
lift} of 
$(\Gamma, p)$ is a map $\tilde{p} \colon 
\Gamma_0 \to \R^3$ such that
\begin{enumerate}
\item
$\pr_\perp \circ \tilde{p} = p$, where $\pr_\perp \colon \R^3 
\to \R^2$ is the orthogonal projection;
\item
for every face $\alpha$ of $(\Gamma,p)$ the points $\{\tilde{p}_i \mid i \in 
\alpha\}$ are coplanar;
\item
the planes of the adjacent faces differ from each other.
\end{enumerate}
A \emph{radial polyhedral lift} of $(\Gamma, p)$ is a map $\tilde{p} \colon 
\Gamma_0 \to \R^3$ that satisfies the above conditions with 1) replaced by
\begin{enumerate}
 \item[1')] $\pr_a \circ \tilde{p} = p$, where $\pr_a \colon \R^3 \setminus 
\{a\} \to \R^2$ is the radial projection from a point $a \notin \R^2$.
\end{enumerate}
\end{dfn}

It turns out that reciprocal diagrams are related to polyhedral lifts and 
to the statics of the framework $(\Gamma, p)$.

A stress $w \colon \Gamma_1 \to \R$ on a framework $(\Gamma, p)$ is called a 
\emph{self-stress} if it resolves the zero load:
\begin{equation}
\label{eqn:SelfStress}
\sum_{j \in \Gamma_0} w_{ij} (p_j - p_i) = 0 \mbox{ for 
all }i \in \Gamma_0.
\end{equation}

\begin{thm}
\label{thm:MCEuc}
Let $(\Gamma, p)$ be a framework in $\R^2$ with a planar $3$-connected 
graph $\Gamma$ and such that for every face $\alpha \in \Gamma_2$ the points 
$\{p_i \mid i \in \alpha\}$ are not collinear. Then the following conditions 
are equivalent:
\begin{enumerate}
\item
The framework has a self-stress that is non-zero on all edges.
\item
The framework has a reciprocal diagram.
\item
The framework has a vertical polyhedral lift.
\item
The framework has a radial polyhedral lift.
\end{enumerate}
\end{thm}
\begin{proof}
1) $\Rightarrow$ 2): From a self-stress $w$ construct a 
reciprocal diagram $(\Gamma^*, m)$ in the following recursive way. Take any 
face $\alpha_0$ and define $m_{\alpha_0} \in \R^2$ arbitrarily. If for some 
face $\alpha$ the point $m_\alpha$ is already defined, then for every $\beta$ 
adjacent to $\alpha$ put
\[
m_\beta = m_\alpha + w_{ij} J(p_j - p_i),
\]
where $J \colon \R^2 \to \R^2$ is the rotation by the angle $\frac{\pi}2$, 
$ij$ is the edge dual to $\alpha\beta$, and the pair $(\alpha\beta, ij)$ 
is consistently oriented. In order to show that this gives a well-defined map 
$m \colon \Gamma_2 \to \R^2$, we need to check that the sum $\sum_{ij} w_{ij} 
J(p_j-p_i)$ vanishes along every closed path in the graph $\Gamma^*$. Viewed as 
a simplicial chain, every closed path is a sum of paths around vertices. The 
sum around a vertex vanishes due to
\eqref{eqn:SelfStress}. By construction, $m_\beta - m_\alpha \perp p_j-p_i$ 
and $m_\alpha \ne m_\beta$ for $\alpha$ and $\beta$ adjacent in $\Gamma^*$, 
thus $(\Gamma^*, m)$ is a reciprocal diagram for $(\Gamma, p)$.


2) $\Rightarrow$ 1): Let $(\alpha\beta, ij)$ be a consistently oriented dual 
pair. Since $m_\beta - m_\alpha \perp p_j - p_i$, there is $w_{ij} \in \R$ such 
that $m_\beta - m_\alpha = w_{ij} J(p_j - p_i)$. The map $w \colon \Gamma_1 \to 
\R$ thus constructed never vanishes and satisfies \eqref{eqn:SelfStress}.


3) $\Rightarrow$ 2): Given a polyhedral lift of $(\Gamma, p)$, let 
$M_\alpha \subset \R^3$ be the plane to which the face $\alpha$ is lifted. 
Since $M_\alpha$ is not vertical, it is the graph of a linear function
$f_\alpha \colon \R^2 \to \R$. Put $m_\alpha = \grad f_\alpha$. For every dual 
pair $(\alpha\beta, ij)$ we have
\[
\tilde{p}_i, \tilde{p}_j \in M_\alpha \cap M_\beta.
\]
This implies that the linear function $f_\alpha - f_\beta$ vanishes along the 
line through $p_i$ and $p_j$, hence
\[
m_\alpha - m_\beta = \grad(f_\alpha - f_\beta) \perp p_i - p_j.
\]

2) $\Rightarrow$ 3): Given a reciprocal diagram $(\Gamma^*, m)$, construct 
a polyhedral lift recursively. Take any $\alpha_0$ and let $f_{\alpha_0} \colon 
\R^2 \to \R$ be any linear function with $\grad f_{\alpha_0} = m_{\alpha_0}$. 
If $f_\alpha$ is defined for some $\alpha$, then define $f_\beta$ for every 
$\beta$ adjacent to $\alpha$ by requiring
\[
\grad f_\beta = m_\beta, \quad f_\beta - f_\alpha = 0 
\text{ on the line } p_ip_j,
\]
where $ij$ is the edge dual to $\alpha\beta$.
These conditions are consistent due to $m_\beta - m_\alpha \perp 
p_j - p_i$. In order to check that the recursion is well-defined, it suffices 
to show that if we start with some $f_\alpha$ and apply the recursion around the 
vertex $i \in \alpha$, then the new $f_\alpha$ will be the same as the old one. 
This is indeed the case because by construction all $f_\beta$ with $i \in 
\beta$ take the same value at $p_i$. A polyhedral lift of $(\Gamma, p)$ is 
obtained by putting $\tilde{p}_i = f_\alpha(p_i)$ for any $\alpha \ni i$.

3) $\Leftrightarrow$ 4): Consider $\R^3$ as an affine chart of $\RP^3$. There 
is a projective transformation $\Phi \colon \RP^3 \to \RP^3$ that restricts to 
the identity on $\R^2 \subset \R^3$ and sends the point $a$ to the point at 
infinity that corresponds to the pencil of lines perpendicular to $\R^2$. (This 
transformation exchanges the plane at infinity with the plane through $a$ 
parallel to $\R^2$.) We have $\pr_a = \pr_\perp \circ \Phi$. Therefore if 
$\tilde p$ is a radial polyhedral lift of $p$, then $\Phi \circ \tilde p$ is an 
orthogonal lift of $p$. Conversely, if $\tilde p$ is an orthogonal lift such 
that $\tilde p_i$ does not lie on the plane through $a$ parallel to $\R^2$, then 
$\Phi^{-1} \colon \tilde p$ is a radial lift. Any orthogonal lift can be 
shifted in the direction orthogonal to $\R^2$ so that its vertices don't lie on 
the plane through $a$ parallel to $\R^2$. Therefore the existence of an 
orthogonal lift is equivalent to the existence of a radial lift.
\end{proof}

\begin{figure}[htb]
\begin{center}
\includegraphics{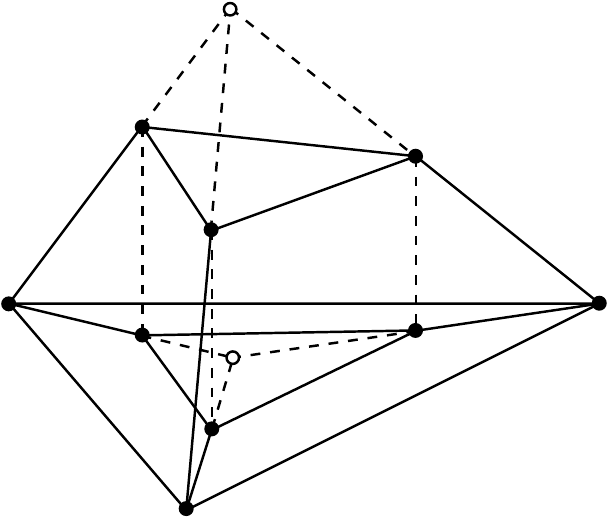}
\end{center}
\caption{A vertical lift of the framework from Example \ref{exl:3Prism}.}
\label{fig:Lift}
\end{figure}

\begin{exl}
\label{exl:Lift3Prism}
The Maxwell-Cremona correspondence allows to prove the rigidity criterium for 
the framework from Example \ref{exl:3Prism}. The lines $a$, $b$, $c$ are 
concurrent if and only if the framework has a vertical lift, see Figure 
\ref{fig:Lift}.
\end{exl}

\begin{rem}
The spaces of self-stresses, reciprocal diagrams, and polyhedral lifts have 
natural linear structures. The correspondences described in the proof of 
Theorem \ref{thm:MCEuc} are linear, see also \cite{CW94}.
\end{rem}

Every graph $\Gamma$ has a geometric realization $|\Gamma|$: assign to the 
vertices points in $\R^3$ in general position, and to the edges the segments 
between those points.
A map $\Gamma_0 \to \R^2$ can be extended to a map $|\Gamma| \to \R^2$ by 
affine 
interpolation. We call this the \emph{rectilinear extension}.
If the rectilinear extension is an embedding, 
then every face of $\Gamma$ (viewed as a cycle of edges) becomes a polygon. In 
this case there is one face that is the union of all the other 
faces; we call it the \emph{exterior face} (the term comes from the 
identification of $\R^2$ with a punctured sphere). The edges of the exterior 
face are called \emph{boundary edges}, all of the other edges are called 
\emph{interior edges}.

\begin{thm}
\label{thm:MCEuc+}
Let $(\Gamma, p)$ be a framework in $\R^2$ with a planar $3$-connected 
graph $\Gamma$ and such that the rectilinear extension of $p$ to $|\Gamma|$ 
provides an embedding of $\Gamma$ into $\R^2$ with convex faces.
Then the following conditions are equivalent:
\begin{enumerate}
\item
The framework has a self-stress that is positive on all interior edges and 
negative on all boundary edges.
\item
The framework has a reciprocal diagram such that for every dual pair 
$(\alpha\beta, ij)$ the pair of vectors $(p_j - p_i, m_\beta - m_\alpha)$ is 
positively oriented if $ij$ is an interior edge and negatively oriented if $ij$ 
is a boundary edge.
\item
The framework has a vertical polyhedral lift to a convex polytope.
\item
The framework has a radial polyhedral lift to a convex polytope.
\end{enumerate}
\end{thm}
\begin{proof}
It suffices to show that the constructions in the proof of Theorem 
\ref{thm:MCEuc} respect the above properties.

1) $\Leftrightarrow$ 2): Since a self-stress is related to a reciprocal diagram 
by the formula $m_\beta - m_\alpha = w_{ij} J(p_j - p_i)$, the pair $(p_j - 
p_i, m_\beta - m_\alpha)$ is positively oriented if and only if $w_{ij} > 0$.

2) $\Leftrightarrow$ 3): Since $m_\beta - m_\alpha = \grad(f_\beta - 
f_\alpha)$, the pairs $(p_j - p_i, m_\beta - m_\alpha)$ for all interior edges 
$ij$ are positively oriented if and only if the piecewise linear function over 
the union of the interior faces defined by $f(x) = f_\alpha(x)$ for $x \in 
\alpha$ is convex. The graph of this function together with the lift of the 
exterior face (that covers the union of the interior faces) form a convex 
polytope.

3) $\Leftrightarrow$ 4): The projective image of a convex polytope (provided 
no point is sent to infinity) is a convex polytope. The orthogonal lift can be 
made disjoint from the plane that is sent to infinity by shifting in the 
vertical direction.
\end{proof}

\begin{rem}
By adding a linear function to an orthogonal polyhedral lift we can achieve 
that the exterior face stays in $\R^2$. A convex polytope of this kind is 
called a convex cap. An example is given on Figure \ref{fig:Lift}.
\end{rem}

\begin{rem}
The only self-intersections of the reciprocal diagram from 
Theorem \ref{thm:MCEuc+} involve the edges $m_{\alpha_0}m_\beta$, where 
$\alpha_0$ is the exterior face of $\Gamma$ (and there is no way to get rid of 
all self-intersections unless $\Gamma$ is the graph of the tetrahedron). The 
reciprocal diagram can be represented without self-intersections by replacing 
every edge $m_{\alpha_0}m_\beta$ with a ray running from $m_\beta$ in the 
direction opposite to $m_{\alpha_0}$. Complexes of this sort are called 
\emph{spider webs} in \cite{WABC88}.

Non-crossing frameworks with non-crossing reciprocals (and thus with some 
non-convex faces) are studied in the article \cite{ORSSSW04}.
\end{rem}

\begin{rem}
The Dirichlet tesselation of a finite point set and the corresponding Voronoi 
diagram are a special case of a framework 
and a reciprocal diagram of the type described in Theorem \ref{thm:MCEuc+}. The 
vertical lift is given by $\tilde{p}_i = (p_i, \|p_i\|^2)$.
The Voronoi diagram represents the reciprocal in the form of a spider web as 
described in the previous remark. A generalization of Dirichlet tesselations 
and Voronoi diagrams are weighted Delaunay tesselations and power diagrams. One 
of the definitions of a weighted Delaunay tesselation is a tesselation that 
possesses a vertical lift to a convex polytope. Thus one can a fifth
equivalent condition to Theorem \ref{thm:MCEuc+}: the framework is a weighted 
Delaunay tesselation. For details see \cite{AKL13}.

In \cite{WABC88} the spider webs were related to planar sections of spatial 
Delaunay tesselations.
\end{rem}

\begin{rem}
Not every convex tesselation and even not every triangulation of a convex 
polygon has a convex polyhedral lift, see \cite[Chapter 7.1]{LRS10} for the 
``mother of all counterexamples''. Those that do are called coherent or regular 
triangulations (more generally, tesselations). There is a 
generalization to higher dimensions, see \cite{LRS10}.
\end{rem}

\subsection{Maxwell-Cremona correspondence in spherical geometry}
\begin{dfn}
Let $(\Gamma, p)$ be a framework in $\Sph^2$ with a planar $3$-connected graph 
$\Gamma$. A \emph{weak reciprocal diagram} for $(\Gamma, p)$ is a framework 
$(\Gamma^*, m)$ in $\Sph^2$ such that
\begin{enumerate}
 \item for every dual pair $(\alpha\beta, ij)$ the geodesics $p_ip_j$ and 
$m_\alpha m_\beta$ are perpendicular;
 \item \label{it:WeakRec} for every incident pair $(\alpha,i)$ the distance 
between 
$m_\alpha$ and $p_i$ is different from $\frac{\pi}2$.
\end{enumerate}
A \emph{strong reciprocal diagram} is defined in the same way except that 
condition \ref{it:WeakRec}) is replaced by
\begin{enumerate}
 \item[2')] for every incident pair $(\alpha,i)$ the distance 
between $m_\alpha$ and $p_i$ is less than $\frac{\pi}2$.
\end{enumerate}
\end{dfn}

The reciprocity conditions can be rewritten as
\begin{equation}
\langle m_\alpha, p_i \rangle \langle m_\beta, p_j \rangle - \langle m_\alpha, 
p_j \rangle \langle m_\beta, p_i \rangle = 0 \label{eqn:ReciprSph1}
\end{equation}
\begin{equation}
\langle m_\alpha, p_i \rangle \ne 0 \label{eqn:ReciprSph2}
\end{equation}
\begin{equation}
\langle m_\alpha, p_i \rangle > 0 \tag{\theequation'}
\end{equation}
The left hand side in \eqref{eqn:ReciprSph1} equals $\langle m_\alpha \times 
m_\beta, p_i \times p_j \rangle$.

\begin{dfn}
Let $(\Gamma, p)$ be a framework in $\Sph^2$ with a planar $3$-connected graph 
$\Gamma$ and such that for every face $\alpha \in \Gamma_2$ the points 
$\{p_i \mid i \in \alpha\}$ are not collinear (that is, don't lie on a great 
circle). A \emph{weak polyhedral lift} of 
$(\Gamma, p)$ is a map $\tilde{p} \colon \Gamma_0 \to \R^3$ such that
\begin{enumerate}
\item
$\tilde p_i = a_i p_i$ for every $i \in \Gamma_0$, where $a_i \ne 0$;
\item
for every face $\alpha \in \Gamma_2$ the points $\{\tilde{p}_i \mid i \in 
\alpha\}$ are coplanar;
\item
the planes of the adjacent faces differ from each other.
\end{enumerate}
A \emph{strong polyhedral lift} is defined similarly but with $a_i > 0$ in 
condition 1.
\end{dfn}

\begin{thm}
\label{thm:MCSph}
Let $(\Gamma, p)$ be a framework in $\Sph^2$ with a 
planar $3$-connected graph $\Gamma$  and such that for every face $\alpha \in 
\Gamma_2$ the points $\{p_i \mid i \in \alpha\}$ are not collinear. Then the 
following conditions are 
equivalent:
\begin{enumerate}
\item
The framework has a self-stress that is non-zero on all edges.
\item
The framework has a weak reciprocal diagram.
\item
The framework has a weak polyhedral lift.
\end{enumerate}
\end{thm}
\begin{proof}
1) $\Rightarrow$ 3): By Lemma \ref{lem:StressExtr}, a self-stress $w$ gives 
rise to a map $\lambda \colon \Gamma_1 \to \R$ such that
\begin{equation}
\label{eqn:Closing13}
\sum_{j \in \Gamma_0} \lambda_{ij} p_j \parallel p_i \quad 
\text{for all }i \in \Gamma_0.
\end{equation}
Pick an $\alpha_0 \in \Gamma_2$ and define $\tilde m_{\alpha_0} \in \R^3$ 
arbitrarily. Define $\tilde m \colon \Gamma_2 \to \R^3$ recursively: if 
$\tilde m_\alpha$ is already defined, then for every $\beta$ adjacent to 
$\alpha$ 
put
\[
\tilde{m}_\beta = \tilde m_\alpha + \lambda_{ij} (p_i \times p_j),
\]
where $(\alpha\beta, ij)$ is a consistently oriented dual pair. Equation 
\eqref{eqn:Closing13} implies that the closing condition around every vertex $i$ 
holds:
\[
\sum_j \lambda_{ij}(p_i \times p_j) = 0.
\]
Thus we have a well-defined map $\tilde m \colon \Gamma_2 \to 
\R^3$ with
\[
\tilde m_\beta - \tilde m_\alpha \parallel p_i \times p_j
\]
for any dual pair $(\alpha\beta, ij)$. In particular, $\tilde m_\beta - \tilde 
m_\alpha \perp p_i$, which implies that for every $i$ there is 
$c_i \in \R$ such that
\[
\langle \tilde m_\alpha, p_i \rangle = c_i
\]
for all $\alpha$ incident to $i$. For a generic initial choice of $\tilde 
m_{\alpha_0}$ we have $c_i \ne 0$ for all $i$. If we put $\tilde p_i = 
\frac{p_i}{c_i}$, then we have
\[
\langle \tilde m_\alpha, \tilde p_i \rangle = 1
\]
for every incident pair $(\alpha, i)$. It follows that for every $\alpha \in 
\Gamma_2$ the points $\{\tilde p_i \mid i \in \alpha\}$ are coplanar and span 
a plane orthogonal to the vector $\tilde m_\alpha$. Due to $\lambda_{ij} \ne 0$ 
for every edge $ij$ the planes of adjacent faces are different, thus we have 
constructed a weak polyhedral lift of $(\Gamma, p)$.

3) $\Rightarrow$ 2): Let $M_\alpha \subset \R^3$ be the plane containing the 
points $\{\tilde p_i \mid i \in \alpha\}$. Since the points $\{p_i \mid i \in 
\alpha\}$ are not collinear, the plane $M_\alpha$ does not pass through the 
origin. Thus it has equation of the form
\[
M_\alpha = \{x \in \R^3 \mid \langle \tilde m_\alpha, x \rangle = 1\}
\]
for some $\tilde m_\alpha \in \R^3$.
In particular, for any dual pair $(\alpha\beta, ij)$ we have
\[
\langle \tilde m_\beta - \tilde m_\alpha, \tilde p_i \rangle = \langle \tilde 
m_\beta - \tilde m_\alpha, \tilde p_j \rangle = 0.
\]
Hence the vector $\tilde m_\beta - \tilde m_\alpha$, and with it the plane 
spanned by $\tilde m_\alpha$ and $\tilde m_\beta$, is perpendicular to the plane 
spanned by $p_i$ and $p_j$.
If we put $m_\alpha = \frac{\tilde m_\alpha}{\|\tilde m_\alpha\|}$, then the 
geodesic $m_\alpha m_\beta$ is perpendicular to the geodesic $p_i p_j$. Since 
$\langle \tilde m_\alpha, \tilde p_i \rangle = 1$, we have $\langle m_\alpha, 
p_i \rangle \ne 0$. Thus $(\Gamma^*, m)$ is a weak reciprocal diagram to 
$(\Gamma, p)$.

2) $\Rightarrow$ 3): Let $(\Gamma^*, m)$ be a weak reciprocal diagram for 
$(\Gamma, p)$. We construct lifts $\tilde m$ and $\tilde p$ such that
\begin{equation}
\label{eqn:Recipr23}
\langle \tilde m_\alpha, \tilde p_i \rangle = 1
\end{equation}
for every incident pair $(\alpha, i)$. The construction is recursive.

Pick $\alpha_0 \in \Gamma_2$ and lift $m_{\alpha_0}$ arbitrarily. Due to 
\eqref{eqn:ReciprSph2}, for every $i \in \alpha_0$ there is a lift $\tilde p_i$ 
of $p_i$ such that $\langle \tilde m_{\alpha_0}, \tilde p_i \rangle = 1$. 
If $\tilde m_\alpha$ is already defined, and 
$\beta$ is adjacent to $\alpha$, then let $ij$ be the edge dual to 
$\alpha\beta$. First determine the lift $\tilde p_i$ from the 
condition \eqref{eqn:Recipr23}, and then determine the lift $\tilde m_\beta$ 
from the same condition with $\beta$ in place of $\alpha$. Note that if we use 
$p_j$ instead of $p_i$, then the result will be the same: due to the 
reciprocity conditions \eqref{eqn:ReciprSph1} and \eqref{eqn:ReciprSph2} we have
\[
\langle \tilde m_\alpha, \tilde p_i \rangle = \langle m_\alpha, \tilde p_j 
\rangle \Rightarrow \langle \tilde m_\beta, \tilde p_i \rangle = \langle \tilde 
m_\beta, \tilde p_j \rangle.
\]
This recursive procedure leads to well-defined lifts $\tilde m$ and $\tilde p$:
going around a vertex $i$ does not change the value of $\tilde m_\alpha$ 
because both the initial and the final values satisfy \eqref{eqn:Recipr23}.

3) $\Rightarrow$ 1): Let $\tilde m \colon \Gamma_2 \to \R^3$ be the map 
constructed during the proof of the implication 3) $\Rightarrow$ 2). As it was 
shown, for every dual pair $(\alpha\beta, ij)$ the non-zero vector $\tilde 
m_\beta - \tilde m_\alpha$ is perpendicular to $p_i$ and $p_j$. Thus we have a 
map $\lambda \colon \Gamma_1 \to \R$ such that
\[
\tilde m_\beta - \tilde m_\alpha = \lambda_{ij} p_i \times p_j.
\]
To determine the sign of $\lambda_{ij}$, we order the vertices so that the pair 
$(\alpha\beta, ij)$ is consistently oriented. Summing around a vertex $i$ of 
$\Gamma$ we obtain
\[
p_i \times \sum_{j \in \Gamma_0} \lambda_{ij} p_j = 0.
\]
Hence $\sum_{j \in \Gamma_0} \lambda_{ij} p_j \parallel p_i$ and by Lemma 
\ref{lem:StressExtr} the map $\lambda$ gives rise to a non-zero self-stress on 
$(\Gamma, p)$.
\end{proof}

We don't know what conditions on a framework and the stress guarantee the 
existence of a strong reciprocal diagram. At least it is necessary that the 
vertices of every face are contained in an open hemisphere. The next theorem 
shows that strong reciprocal diagrams correspond to strong polyhedral lifts.

\begin{thm}
Let $(\Gamma, p)$ be a framework in $\Sph^2$ as in Theorem \ref{thm:MCSph}. 
Then the following conditions are equivalent:
\begin{enumerate}
 \item The framework has a strong reciprocal diagram.
 \item The framework has a strong polyhedral lift.
\end{enumerate}
\end{thm}
\begin{proof}
In the proof of 3) $\Rightarrow$ 2) in Theorem \ref{thm:MCSph}, note that for a 
strong lift $\tilde p$ the equation $\langle \tilde m_\alpha, \tilde p_i 
\rangle = 1$ implies $\langle m_\alpha, p_i \rangle > 0$, so that the 
reciprocal diagram constructed from a strong lift is strong itself.

In the proof of 2) $\Rightarrow$ 3) in Theorem \ref{thm:MCSph}, lift 
$m_{\alpha_0}$ strongly (that is scale it by a positive factor). Condition 
\eqref{eqn:Recipr23} implies that all $p_i$ with $i \in \alpha_0$ are also 
lifted strongly. The recursion propagates the strong lift to all $m_\beta$ 
and~$p_j$.
\end{proof}

As in the Euclidean case (see the paragraph before Theorem \ref{thm:MCEuc+}), a 
spherical framework defines a geodesic extension, that is a map $|\Gamma| \to 
\Sph^2$ that sends every edge to an arc of a great circle. A geodesic extension 
is called a convex embedding of $\Gamma$ if it is an embedding and every face 
is a convex spherical polygon.

\begin{thm}
Let $(\Gamma, p)$ be a framework in $\Sph^2$ with a planar $3$-connected graph 
$\Gamma$ and such that its geodesic extension is a convex embedding. Then the 
following conditions are equivalent:
\begin{enumerate}
\item
The framework has a self-stress that is positive on all edges.
\item
The framework has a strong reciprocal diagram that embeds $\Gamma^*$ in 
$\Sph^2$ with convex faces.
\item
The framework has a strong lift to a convex polyhedron.
\end{enumerate}
\end{thm}
\begin{proof}
1) $\Rightarrow$ 3): In the proof of the corresponding implication in Theorem 
\ref{thm:MCSph} we have $\lambda_{ij} > 0$ for all edges $ij$. This implies 
that as we go around a vertex $i$, the vertices $\tilde m_\alpha$ for all 
$\alpha$ adjacent to $i$ form a convex polygon. The union of these polygons is 
the boundary of a convex polyhedron that contains $0$ in its interior. Its 
polar dual is a strong lift of $(\Gamma, p)$.

3) $\Rightarrow$ 2): A convex polyhedron that is a strong lift of $(\Gamma, p)$ 
contains $0$ in the interior. Thus its polar dual is also a convex polyhedron. 
The projection of the $1$-skeleton of the dual is a strong reciprocal diagram 
with convex faces.

2) $\Rightarrow$ 1): In a strong reciprocal diagram with convex faces the 
geodesics $m_\alpha m_\beta$ and $p_i p_j$ that correspond to a dual pair are 
consistently oriented.
When we lift such a diagram as in the proof of 2) 
$\Rightarrow$ 3) $\Rightarrow$ 1) in Theorem \ref{thm:MCSph}, we obtain real 
numbers $\lambda_{ij} > 0$ that provide a positive self-stress on $(\Gamma, p)$.
\end{proof}

The latter version of the spherical Maxwell-Cremona correspondence was 
described in \cite{Lov01}.

\begin{rem}
As in the Euclidean case, not every convex tesselation of the sphere has a 
convex polyhedral lift. The corresponding theory predates the theory of 
regular triangulations in the Euclidean space and was developed by Shephard 
\cite{She71} and McMullen \cite{McM73}. See also \cite{FI17}.
\end{rem}

\subsection{Maxwell-Cremona correspondence in hyperbolic geometry}
Let $(\Gamma, p)$ be a framework in $\H^2$ with a planar $3$-connected graph 
$\Gamma$. A \emph{reciprocal diagram} is a framework $(\Gamma^*, m)$ in 
$\H^2$ such that 
for every dual pair $(\alpha\beta, ij)$ the geodesics $m_\alpha m_\beta$ and 
$p_ip_j$ are perpendicular. In terms of the Minkowski scalar product this means
\[
\langle m_\alpha, p_i \rangle \langle m_\beta, p_j \rangle - \langle m_\alpha, 
p_j \rangle \langle m_\beta, p_i \rangle = 0.
\]

\begin{rem}
The above criterion of orthogonality of $m_\alpha m_\beta$ and $p_i p_j$ as 
well as its spherical analog \eqref{eqn:ReciprSph1} can be reformulated as 
follows.
Diagonals in a spherical or hyperbolic quadrilateral with the side lengths 
$a, b, c, d$ in this cyclic order are orthogonal if and only if
\[
\cos_\X a \cos_\X c = \cos_\X b \cos_\X d.
\]
The diagonals of a Euclidean quadrilateral are orthogonal if and only if $a^2 
+ c^2 = b^2 + d^2$.
\end{rem}

\begin{dfn}
Let $(\Gamma, p)$ be a framework in $\H^2$ with a planar $3$-connected graph 
$\Gamma$ and such that for every face $\alpha \in \Gamma_2$ the points 
$\{p_i \mid i \in \alpha\}$ are not collinear. A \emph{polyhedral lift} of 
$(\Gamma, p)$ is a map $\tilde{p} \colon \Gamma_0 \to \R^3$ such that
\begin{enumerate}
\item
$\tilde p_i = a_i p_i$ for every $i \in \Gamma_0$, where $a_i > 0$;
\item
for every face $\alpha \in \Gamma_2$ the points $\{\tilde{p}_i \mid i \in 
\alpha\}$ are contained in a space-like plane;
\item
the planes of the adjacent faces differ from each other.
\end{enumerate}
\end{dfn}

\begin{thm}
\label{thm:MCHyp}
Let $(\Gamma, p)$ be a framework in $\H^2$ with a 
planar $3$-connected graph $\Gamma$  and such that for every face $\alpha \in 
\Gamma_2$ the points $\{p_i \mid i \in \alpha\}$ are not collinear. Then the 
following conditions are equivalent:
\begin{enumerate}
\item
The framework has a self-stress that is non-zero on all edges.
\item
The framework has a reciprocal diagram.
\item
The framework has a polyhedral lift.
\end{enumerate}
\end{thm}
\begin{proof}
1) $\Rightarrow$ 3): Proceed as in the proof of Theorem \ref{thm:MCSph} to 
obtain a map $\tilde m \colon \Gamma_2 \to \R^3$ such that
\[
\tilde m_\beta - \tilde m_\alpha \parallel p_i \times p_j
\]
(with the Minkowski cross-product) for every dual pair $(\alpha\beta, ij)$. By 
changing the position of $\tilde m_{\alpha_0}$ and scaling down the self-stress 
$w$ we can achieve that all $\tilde m_\alpha$ belong to the upper half of the 
light cone. Then the planes $\langle \tilde m_\alpha, x \rangle = -1$ bound a 
polyhedron with space-like faces that is a polyhedral lift of $(\Gamma, p)$.

3) $\Rightarrow$ 2): Similarly to the proof of Theorem \ref{thm:MCSph}, let 
$\langle \tilde m_\alpha, x \rangle = -1$ be an equation of the plane 
containing the points $\{\tilde p_i \mid i \in \alpha\}$. Since these planes 
are space-like, $\tilde m_\alpha$ are time-like, and since $\tilde p_i$ belongs 
to the upper half of the light cone, $\tilde m_\alpha$ also does. Hence 
$(\Gamma^*, m)$ is a reciprocal diagram in $\H^2$.

2) $\Rightarrow$ 3): The proof is the same as in Theorem \ref{thm:MCSph}, we 
lift $(\Gamma, p)$ and $(\Gamma^*, m)$ recursively and at the same time.

3) $\Rightarrow$ 1): Also the same as in Theorem \ref{thm:MCSph}, but with the 
Minkowski cross-product instead of the Euclidean.
\end{proof}

For a framework $(\Gamma, p)$ in $\H^2$ the geodesic extension $|\Gamma| 
\to \H^2$ is an analog of the rectilinear extension in the Euclidean case: 
an edge $ij$ of $\Gamma$ is mapped to the geodesic segment $p_ip_j$. If the 
geodesic extension is an embedding, then we define the interior and exterior 
faces and the interior and boundary edges as in the Euclidean case, see the 
paragraph before Theorem \ref{thm:MCEuc+}.

For a consistently oriented dual pair $(\alpha\beta, ij)$ we say that the lines 
$p_ip_j$ and $m_\alpha m_\beta$ are consistently oriented if the directed line 
$m_\alpha m_\beta$ is obtained from the directed line $p_ip_j$ through rotation 
by $\frac\pi{2}$ around their intersection point.

\begin{thm}
Let $(\Gamma, p)$ be a framework in $\H^2$ with a planar $3$-connected 
graph $\Gamma$ and such that the geodesic extension of $p$ to $|\Gamma|$ 
provides an embedding of $\Gamma$ into $\H^2$ with convex faces.
Then the following conditions are equivalent:
\begin{enumerate}
\item
The framework has a self-stress that is positive on all interior edges and 
negative on all boundary edges.
\item
The framework has a reciprocal diagram such that for every consistently 
oriented dual pair $(\alpha\beta, ij)$ the lines $p_ip_j$ and 
$m_\alpha m_\beta$ are consistently oriented if $ij$ is an interior edge and 
non-consistently oriented if $ij$ is a boundary edge.
\item
The framework has a polyhedral lift to a convex polytope in the Minkowski space.
\end{enumerate}
\end{thm}
\begin{proof}
The proof consists in checking that the constructions in the proof of Theorem 
\ref{thm:MCHyp} respect the above properties.
\end{proof}

\begin{rem}
A variant of the Maxwell-Cremona theorem for hyperbolic frameworks uses an 
orthogonal polyhedral lift to the co-Minkowski space instead of a radial 
polyhedral lift to the Minkowski space described above. For details on the 
co-Minkowski space see \cite{FS}.
\end{rem}

\begin{rem}
If we allow the faces of the polyhedral lift to be time-like or light-like, 
then the vertices of the corresponding reciprocal diagram become de Sitter or 
ideal. Since the reciprocity is a symmetric notion, it is natural to allow 
de Sitter and ideal positions for the vertices of the framework as well. This 
puts us into the more general context of hyperbolic-de Sitter frameworks or 
point-line-horocycle frameworks, see Section \ref{sec:PointLine}.
\end{rem}

%
%
%
%
%
%
%
%
%
%
%
%
%
%
%


\end{document}